\documentclass[smallextended]{svjour3}     
\smartqed  
\usepackage{graphicx}

\usepackage{savesym}
\usepackage{amsmath}
\usepackage{graphics}
\usepackage{subfigure}
\usepackage{amssymb}
\usepackage{txfonts}      

\RequirePackage{natbib}
\RequirePackage{hypernat}
\usepackage{url}
\newcommand{\href}[2]{{#2}}
\RequirePackage{ifthen}

 \newcommand{\Ts}{\textstyle}
\newcommand{\Ss}{\scriptstyle}  \newcommand{\SSs}{\scriptscriptstyle}
\newcommand{\ssr}{\rm\scriptscriptstyle}
\newcommand{\req}{\relax}
\newcommand{\rfia}[1]{\makebox[\parindent][l]{%
                     \makebox[0em][r]{\rm(}\sf#1\rm)}}
\newcounter{ABCcB}
\newcommand{\theABCcC}{\alph{ABCcB}}

\newcommand{\Ew}{\mathop{\rm {{}E{}}}\nolimits} 
\newcommand{\sign}{\mathop{\rm sign}\nolimits}    
\newcommand{\ve}{\varepsilon}

\newcommand{\Lo}{\mathop{\rm {{}o{}}}\nolimits}
\newcommand{\LO}{\mathop{\rm {{}O{}}}\nolimits}

\newcommand{\B}{\mathbb B}

\newcommand{\R}{\mathbb R}

\newcommand{\N}{\mathbb N}
\newcommand{\Jc}{\mathop{\bf\rm{{}I{}}}\nolimits}
\newcommand{\EM} {{\mathbb I}}

\newcommand{\iid}{\mathrel{\stackrel{\ssr i.i.d.}{\sim}}}

\newcommand{\Tfrac}[2]{\textstyle\frac{#1}{#2}}
\newcommand{\Tint}{\mathop{\Ts\int}\nolimits}
\newcommand{\Tsum}{\mathop{\Ts\sum}\nolimits}
   
\newlength{\SyW}   \newlength{\msu}  \msu=\mathsurround 

\newcommand{\wto}{\mathrel{\mathsurround0em \mbox{$\longrightarrow$}%
  \llap{\settowidth{\SyW}{$\longrightarrow$}
  \raisebox{-.15ex}{\makebox[\SyW]{\scriptsize\rm w}}}}}

\newtheorem{Thm}{Theorem}
\newtheorem{Prop}[Thm]{Proposition}
\newtheorem{Lem}[Thm]{Lemma}
\newtheorem{Rem}[Thm]{Remark}
\newtheorem{Cor}[Thm]{Corollary}
\newtheorem{Def}[Thm]{Definition}

\numberwithin{equation}{section}
\numberwithin{Thm}{section}

\newcounter{ABCc}
\renewcommand{\theABCc}{\alph{ABCc}}
\newenvironment{ABC}{\begin{list}{
  \rfia{\theABCc}}{\usecounter{ABCc} \topsep 0ex \partopsep 0ex \itemsep0ex
  \parsep=\parskip \leftmargin 0em \rightmargin 0em \itemindent=\parindent
  \listparindent=\parindent  \labelsep 0.5em \labelwidth 0.5em }}{\end{list}}

\newcommand{\ReprintofEdition}[1]{Reprint of the #1 edition.}
 \journalname{}
\begin{document}

\ifx\blinded\undefined
\title{Higher order asymptotics for the MSE of the sample median on shrinking neighborhoods}

\author{Peter Ruckdeschel}

\institute{P. Ruckdeschel \at
              Fraunhofer ITWM, Department of Financial Mathematics, \\
              Fraunhofer-Platz 1, 67663 Kaiserslautern, Germany\\
              and Dept.\ of Mathematics, University of Kaiserslautern,\\
              P.O.Box 3049, 67653 Kaiserslautern, Germany \\
              \email{peter.ruckdeschel@itwm.fraunhofer.de}\\           
}
\else
\title{Higher order asymptotics for the MSE of the sample median on shrinking neighborhoods}
\author{}\institute{}
\fi
\date{Received: date / Accepted: date}
\maketitle
%
%
%
%
\begin{abstract}
We provide an asymptotic expansion of the maximal mean squared error (MSE)
of the sample median to be attained on shrinking gross error neighborhoods
about an ideal central distribution.
More specifically, this expansion comes in powers of $n^{-1/2}$,
for $n$ the sample size, and uses a shrinking rate of $n^{-1/2}$ as well.
This refines corresponding results of first order asymptotics
to be found in \citet{Ri:94}.\\
In contrast to usual higher order asymptotics, we do not approximate distribution
functions (or densities) in the first place, but rather expand the risk directly.\\
Our results are illustrated by comparing them to the results of a simulation study and
to numerically evaluated exact MSE's in both ideal and contaminated situation.
\end{abstract}
%
%
\section{Motivation/introduction}
\subsection{Simulations as starting point}
This paper 
was initiated by a
simulation study performed by the present author
and M.~Kohl at Bayreuth university in 2003 for a
presentation to be given in the framework of an
invitation by S.~Morgenthaler to EPF Lausanne.
The goal was to investigate the
finite sample behavior of procedures,
which are distinguished as (first order) asymptotically
optimal in infinitesimal robust statistics as to
maximal MSE on $\sqrt{n}$-shrinking (convex-contamination)
neighborhoods.
The results of this study for one dimensional Gaussian
location were so promising already for sample sizes $n$ down
to about 20 that it seemed worthwhile to dig a little deeper.
At closer inspection of the results, we realized
that the approximation quality of this first order asymptotics
could even be much enhanced
down to sample sizes $n=5$ and $10$ if we ignored samples
where more than half the sample stemmed from a contamination.

Asymptotically, in our shrinking neighborhood setting,
 such events carry positive, but exponentially-fast
 decaying probability for any sample size.
\subsection{Description of the main result and discussion}
These empirical findings can indeed be substantiated by theory,
deriving a uniform higher order asymptotic expansion
for the MSE on correspondingly thinned out neighborhoods
for the median, location M-estimators for monotone scores,
and one-step-constructions.
This paper deals with the median case. It is separated from
more general location M-estimators, as the techniques used there
are not available for the median due to a failure of a Cram\'er
condition. Moreover, in higher order asymptotics, even for the
ideal model, differences appear between diverse variants of
the median used for even sample size, a fact which
to the author's knowledge has not been spelt out in detail so far.

Denoting by $\tilde {\cal U}_n(r)$ the neighborhoods thinned out
by cutting away samples with more than $50\%$ contaminations,
and by $M_n$ a suitable variant of the median,
the main result of this paper is
\begin{equation}
\boldmath
\sup_{G^{(n)}\in \tilde {\cal U}_n(r)}\,n\,[{\rm MSE}(M_n,G^{(n)})]=\Tfrac{1}{4 f_0^2}\Big(
(1+r^2)+\Tfrac{r}{\sqrt{n}} a_1 +\Tfrac{1}{n}a_2\Big) +\Lo(\Tfrac{1}{n})\label{firstres}
\end{equation}
with $a_1$ and $a_2$ certain functions in $r, f_0, f_1$ and, for $a_2$, in $f_2$, where
$r$ is the contamination radius and $f_i$ are the values of the
ideal density $f$ and its first and second derivatives evaluated at the ideal median.

As a byproduct of the main result, we are able to give necessary and sufficient
conditions for a contamination to attain the RHS of \eqref{firstres};
it is astonishingly small: all mass of the contaminating
measures has essentially to be concentrated either left of $-{\rm const}\;\sqrt{\log(n)/n}$
or right of ${\rm const}\;\sqrt{\log(n)/n}$.

In formula \eqref{firstres}, we already recognize the following features of the result:\\
The speed of convergence of the ${\rm MSE}$ to its asymptotic value is uniform
on the whole (modified) neighborhood, and is one order faster in the ideal model;
besides, we may work with the original risk (instead of using a modification
as usually).

The expansion in powers of $n^{-1/2}$, in the ideal model with first
correction term at $n^{-1}$, comes surprising: Using first order von Mises
expansions (compare \eqref{ALEd} below), in the context of quantiles
(comprising the sample median), it can be shown by means of
Bahadur-Kiefer representations that the approximation
error of this expansion is an exact $\LO_{F^n}(n^{-1/4})$---cf.\
e.g.\ \citet{Ju:Se:96}. So one would expect that under  
uniform integrability, 
the first correction term in an expansion
of type \eqref{firstres} in the ideal model would be of order
$n^{-1/4}$, too. In fact, \citet{Dut:73} showed that the $L_2$-norm
of the remainder is of exact order $\LO(n^{-1/4})$ in our scaled up setup.
These results are no contradiction to \eqref{firstres}, though, as
 the remainder of course is correlated with the asymptotic linear terms.
We still do not see however how  Bahadur-Kiefer representations
translate into \eqref{firstres}.
\\
In any case, the approximations of type \eqref{firstres} prove very reasonable when
compared to both numerical and simulated values of the {\rm MSE} for finite $n$.\\
With the same techniques, we deal with a number of variants of the median for even sample sizes, and specialize these
results for the case of $F={\cal N}(0,1)$. For odd sample size, we also derive asymptotics of this kind for the variance and bias
separately.\\

In proving the main theorem, we use {\tt MAPLE} at large to
calculate tedious, lengthy asymptotic expansions which
are hardly presentable in the framework of an article.

\subsection{Setup}
We study the accuracy of the sample median as a location estimator on shrinking neighborhoods:
We work in an ideal location model ${\cal P}=\{P_{\theta}\,|\,\theta\in\R\}$ with location parameter $\theta$, observations $X_i\iid P_{\theta}$
and errors $u_i$ given by
\begin{equation}\label{idmod}
  X_i=u_i+\theta,\qquad u_i\iid F
\end{equation}
Due to translation equivariance in the location model we may limit ourselves to $\theta=0$.
We assume that
\begin{equation}\label{medunb}
F(0)=1/2
\end{equation}
i.e.; the location parameter $\theta$ equals the median of the observation distribution,
and that $F$ around $0$ admits a Lebesgue density $f$ with Taylor expansion about $0$ as
\begin{equation}\label{tayl1}
f(x)=f_0+f_1x+\Tfrac{1}{2}\,f_2x^2+\LO(x^{2+\delta_0}),\qquad f_0>0
\end{equation}
for some $\delta_0>0$. Furthermore,
Finally, we assume that there is a $\delta>0$ 
such that
\begin{equation}\label{abkling}
\int |x|^\delta f(x)\,dx<\infty
\end{equation}

\begin{Rem}\rm\small \label{slowdecay}
\begin{ABC}
\item Condition~\eqref{abkling} is taken from \citet{Jur:82} and is both necessary and sufficient for finiteness of
$\Ew_F|M_n|^\gamma$ for any $\gamma>0$, where $M_n$ is the sample median ${\rm Med}_n$  to odd sample size $n$, respectively
any variant of the sample median considered in this paper for $n$ even ---for a proof see subsection~\ref{slowdecayp}.
\item By the H\"older-inequality, $\int |x|^\eta \,F(dx)<\infty$ for each $0<\eta\leq \delta$, so we may assume that $\delta<1$.
\end{ABC}
\end{Rem}
We want to assess both variance and bias simultaneously, so we work with the setup
of shrinking neighborhoods as in \citet{Ri:94}, i.e.\ as deviations from the ideal model~(\ref{idmod}), we consider the set ${\cal Q}_n={\cal Q}_n(r)$ of
distributions
\begin{equation}\label{neighb}
  G^{(n)}:={\Ts\bigotimes\limits_{i=1}^n} G_{n,i},\qquad G_{n,i}=(1-r/\!\sqrt{n}\,) F + r/\!\sqrt{n} \,H_{n,i}
\end{equation}
for arbitrary, uncontrollable contaminating distributions $H_{n,i}$.
As usual, we interpret $G^{(n)}$ as the distribution
of the vector $(X_i)_{i\leq n}$ with components
\begin{equation}\label{Uidef}
  X_i:=(1-U_i)X_i^{\SSs \rm id}+U_i X_i^{\SSs \rm di}
\end{equation}
for $X_i^{\SSs \rm id}$, $U_i$, $X_i^{\SSs \rm di}$ stochastically independent, $X_i^{\SSs \rm id}\iid F$, $U_i\iid{\rm Bin}(1,r/\sqrt{n})$,
and $(X_i^{\SSs \rm di})\sim H_n$ for some arbitrary $H_n \in{\cal M}_1(\B^n)$.
In this setup the median can be understood as an asymptotically linear estimator with
influence curve $\psi_{\rm Med}$, allowing the expansion
\begin{equation}\label{ALEd}
{\rm Med}_n=\Tfrac{1}{n} \sum_{i=1}^{n} \psi_{\rm Med}(X_i)+\Lo_{F^n}(n^{-1/2}),\qquad \psi_{\rm Med}(x)=\frac{\sign(x)}{2f_0}
\end{equation}
--- c.f. \citet[Thm.~1.5.1.]{Ri:94}. Using a clipped version of the quadratic loss function for the estimator $S_n={\rm Med}_n$,
\begin{equation}\label{clipMSE}
{\rm MSE}_M(S_n,G):=\Ew_{G}(\min(n \,S_n^2,M)),
\end{equation}
we may proceed as outlined in \citet[p.~207]{Ri:94},
and obtain
\begin{equation} \label{modifloss}
 \lim_{M\to\infty}\lim_{n\to\infty} \sup\nolimits_{G^{(n)}}\,n\,[{\rm MSE}_M({\rm Med}_n,G^{(n)})]=(4 f_0^2)^{-1}(1+r^2)
\end{equation}
In this paper we want to (a) examine the approximation quality of \eqref{modifloss}, spelling out higher order error terms
and (b) discuss the accuracy of this approximation by comparing it to both numerical evaluations of the exact ${\rm MSE}$'s
and simulation results.\\
In contrast to usual higher order asymptotics, instead of giving approximations to distribution functions (or densities)
by Edgeworth expansions or using saddlepoint techniques---cf.\  e.g.\ \citet{Fi:Ro:90}%
---we proceed by expanding the risk directly.\\
As indicated, for (a) we need to modify the neighborhoods, admitting only such samples
where less than half of the sample is contaminated, that is $\sum U_i< n/2$ in (\ref{Uidef}).
As a side effect of this modification, we will (c) get rid of the somewhat artificial, as statistically unmotivated, modification
of the loss function by clipping \eqref{modifloss}, which is common in asymptotic statistics,  see, among others,
\citet{LC:86}, \citet{Ri:94}, \citet{BKRW:98}, \citet{VdW:98}.\\
%
\subsection{Organization of the paper}
We start with discussing the mentioned modification in detail in section~\ref{modifsec}. The central theoretical result,
Theorem~\ref{medthm} is presented in section~\ref{mainsec}. We then present some ramifications in
section~\ref{ramifsec} covering in particular several variants of the sample median for even sample size
in Theorem~\ref{medeventhm} and Proposition~\ref{aijspec};
also corresponding expansions are given for bias and variance separately in Proposition~\ref{BiasVar}.
Results are spelt out in the special case of $F={\cal N}(0,1)$ in Corollaries~\ref{medthmnormal} and \ref{medthmnormid}.
These theoretic findings are illustrated with
numerical and simulated results in section~\ref{illustrsec}. In the appendix section~\ref{proofsec},
we give proofs to all our assertions.

\begin{Rem}\rm\small
It took me some time to write things up in a readable fashion:
In order not to slay down the reader with vast number of terms,
in the proof section, we instead describe verbally how we got them
referring to a corresponding {\tt MAPLE} script available in the
internet. To give you an idea of how tedious terms become, we have included
a page of {\tt MAPLE} output on page~\pageref{abschrFig}
as a horrifying example.
\end{Rem}
\section{Modification of the shrinking neighborhood setup}\label{modifsec}\req
The shrinking--neighborhood setup guarantees
uniform weak convergence of any as.~linear estimator (ALE) to corresponding normal distributions 
on a representative subclass of the neighboring distributions of form \eqref{neighb} --- those
distributions induced by simple perturbations $Q_n(\zeta,t)$, see \citet[p.~126]{Ri:94}.\\
By the continuous mapping theorem, uniform weak convergence of
these ALE's on ${\cal Q}_n$ entails uniform convergence of the risk for continuous, bounded loss functions like the clipped version
of the MSE \eqref{clipMSE}. 
However, even this (uniform) weak convergence does not entail convergence of the risk for an unbounded loss function like the
(unmodified) MSE in general, as we show in the following proposition:
\subsection{Convergence failure of the MSE for the median}
\begin{Prop}\label{badmed}
  Let ${\cal P}$ be the location model from \eqref{idmod} with $f(0)>0$ and let ${\rm Med}_n$ be the sample median.
  Then for each odd $n=2m+1$ and to any $C>0$ there is an $x_0\in\R$  such that with
$  G^{(n)}_0=[(1-\Tfrac{r}{\sqrt{n}}) F+\Tfrac{r}{\sqrt{n}}\,\Jc_{\{x_0\}}]^n$
  \begin{equation} \label{>C}
   {\rm MSE}({\rm Med}_{n},{G^{(n)}_{0}})>C
  \end{equation}
  although, uniformly in ${\cal Q}_{n}$,
  \begin{equation}
\sqrt{n}\,\Big({\rm Med}_n-\Tfrac{1}{n} \Tsum_{i=1}^n \Tint \psi_{\rm\SSs Med}\,dG_{n,i} \Big)\circ G^{(n)}\wto {\cal N}(0,(2f(0))^{-2})
  \end{equation}
\end{Prop}
\subsection{Modification of the shrinking neighborhood setup}
In view of proposition~\ref{badmed}, a straightforward modification for finite $n$
consists in permitting only such realizations of $U_1,\ldots,U_{n}$, where $K=\sum U_i < n/2$.
More precisely, for $0<\ve<1/2$, we consider the neighborhood system $\tilde {\cal Q}_n(r,\ve)$ of conditional distributions
%
\begin{equation}\label{contdef}
G^{(n)}={\cal L}\{[(1-U_i)X_i^{\SSs \rm id}+U_i X_i^{\SSs \rm di}]_i\,\Big|\,\limsup_n \frac{1}{n}\sum U_i\leq {\ve}\,\}
\end{equation}
%
%
%
%
If we apply the Hoeffding inequality (\citet[Thm.~2]{Hoef:63}) 
%
to $K=\sum_{i=1}^n U_i$ 
for the switching variables $U_i$ from \eqref{Uidef}, 
we obtain
\begin{eqnarray}
\!P(K>m)\!
&\leq&
\exp\Big(-2n (\ve-\Tfrac{r}{\sqrt{n}})^2 \Big)
\end{eqnarray}
which shows the announced asymptotic exponential negligibility of this modification.
Thus all results on convergence in law of the shrinking neighborhood
setup are not affected when passing from
${\cal Q}_n(r)$ to $\tilde {\cal Q}_n(r)$: 
Let $B_n:=\{K< n/2\}$. Then we have for any $\delta>0$ and any sequence of events $A_n$
\begin{eqnarray*}
P(A_n\,|\,B_n)&=&P(A_n\cap B_n)/P(B_n)
=P(A_n)(1+\LO(e^{-2n\ve^2/(1+\delta)}))
\end{eqnarray*}
\subsection{Connection to the breakdown point}
Our definition of the neighborhood $\tilde {\cal Q}_n(r)$ combines the shrinking neighborhood concept,
which will eventually dominate, with a sample--wise restriction;
for some number $\ve\in(0,1)$ depending on the estimator $S_n$, we only allow for samples where
strictly less than $\ve({S_n})n$ observations are contaminated.
This number $\ve({S_n})$ is actually 
just the finite sample ($\ve$-contamination) breakdown point of an estimator $S_n$ introduced by \citet{Do:Hu:83}.  
\\
Thus the concept easily generalizes from the location case to other models: Let ${\cal P}=\{P_{\theta},\,\theta\in\Theta\}$ be a parametric model and
$X_i^{\SSs \rm id}$ be $\R^k$-valued observations distributed according to the ideal situation $P_{\theta}$. We are interested in the question
whether for some given estimator $S_n$, we have uniform convergence of the risk $\int \ell(S_n-\theta)\,dQ_n$ for some loss $\ell\ge 0$ on some
(thinned out) neighborhood or not. To this end we define $\tilde{\cal Q}_n (r,\ve)$ analogously to \eqref{contdef}.
Assume that there is some $\bar \ve>0$ such that for each $n\in\N$ and $k\leq \bar k:=\ulcorner n\bar \ve \urcorner -1$
\begin{equation}
\begin{array}{ll}
\ve_0(S_n):=\inf \Big\{\ve^{\ast}(X_{n-k},S_n)\;\big|\;&X_{n-k}=(x_1,\ldots,x_{n-k})\;\mbox{a possible sample}\\
&\mbox{ configuration,}\; k\leq \bar k\Big\}\ge \bar \ve>0
\end{array}
\end{equation}
where $\ve^{\ast}(X,S)$ is the finite sample ($\ve$-contamination) breakdown point of $S$ at sample $X$.
Then, by an analogue argument to that of Proposition~\ref{badmed}, the following proposition holds:
\begin{Prop}
Assume that $\ell$ is unbounded. Then for any $\ve\ge \ve_0(S_n)$ and $r>0$, the maximal risk of $S_n$
on $\tilde{\cal Q}_n (r,\ve)$ is unbounded; in particular, uniform convergence of the risks does not hold.
\end{Prop}

The other direction of this connection is more involved and is deferred to a subsequent paper.
Under slight additional assumptions, for suitably constructed ALEs to bounded influence
curves and for continuous, polynomially growing loss functions, uniform convergence of the risk
holds on $\tilde{\cal Q}_n (r,\ve)$ for any $\ve <\bar \ve$.
Note that this thinning out for continuous loss functions $\ell$ is not needed if $\ell$ is bounded.
\section{Higher order asmyptotics for the MSE of the sample median}\label{mainsec}\req
%
For $H\in{\cal M}_1(\B^n)$ and an ordered set of indices $I=(1\leq i_1<\ldots<i_k\leq n)$ denote $H_{I}$ the marginal of
$H$ with respect to $I$.
\begin{Def}\label{thinout}
Consider three sequences $c_n$, $d_n$, and   $\kappa_n$ in $\R$, in $(0,\infty)$, and in $\{1,\ldots,n\}$, respectively.
We say that the sequence $(H^{(n)})\subset{\cal M}_1(\B^n)$ is {\em $\kappa_n$--concentrated left [right] of $c_n$ up to $\Lo(d_n)$}, if
for each sequence of ordered sets $I_n$ of cardinality $i_n\leq \kappa_n$
\begin{equation}
1-H^{(n)}_{I_n}\big((-\infty;c_n]^{i_n}\big)=\Lo(d_n)
\qquad
\Big[\,
1-H^{(n)}_{I_n}\big((c_n,\infty)^{i_n}\big)=\Lo(d_n)\,
\Big]
\end{equation}
\end{Def}
\subsection{Main theorem}
\begin{Thm}\label{medthm}\begin{ABC}
\item
In the location model \eqref{idmod} with ideal central distribution $F$  of
finite Fisher information of location, we assume conditions \eqref{tayl1} to \eqref{abkling}.
Then for any $\ve<1/2$, for $G^{(n)}$ varying in $\tilde {\cal Q}_n(r,\ve)$
of \eqref{contdef} it holds
\begin{equation}
\sup_{G^{(n)}}\,n\,[{\rm MSE}({\rm Med}_n,G^{(n)})]=\frac{1}{4 f_0^2}\big(1+r^2+\frac{r}{\sqrt{n\,}\,}\,a_1+\frac{1}{n}\,a_2+\Lo(1/n)\big)\label{Mall}
\end{equation}
for
\begin{eqnarray}
a_1&=& 2(1+r^2)+\frac{r^2+3}{2}\frac{|f_1|}{f_0^2} \label{Malla1}\\
a_2&=&(-2+3r^2+3r^4)+\frac{3r^2(3+r^2)}{2}\frac{|f_1|}{f_0^2}-\frac{3+6r^2+r^4}{12}\frac{f_2}{f_0^3}+ \nonumber\\
&&\qquad +\frac{5(3+6r^2+r^4)}{16}\frac{f_1^2}{f_0^4} \label{Malla2}
\end{eqnarray}
\item The maximal contamination is achieved by any sequence of contaminating measures $(H_{n})$,
such that for $k_1>1$ and $k_2>\sqrt{5/2}$, 
and for 
$\kappa_n=\ulcorner k_1r\sqrt{n}\; \urcorner$, eventually in $n$, 
either
\begin{equation}
  \label{contbed1}
(H_n) \mbox{ is  $\kappa_n$--concentrated left of $-\Tfrac{k_2}{f_0}\sqrt{\log(n)/n}$ up to }\Lo(n^{-1})
\end{equation}
or
\begin{equation}
  \label{contbed2}
(H_n) \mbox{ is  $\kappa_n$--concentrated right of $\Tfrac{k_2}{f_0}\sqrt{\log(n)/n}$ up to }\Lo(n^{-1})
\end{equation}
More precisely, if $f_1<0$ [$f_1>0$], the maximal MSE is achieved up to $\LO(n^{-2})$ by contaminations according to \eqref{contbed1}
[\eqref{contbed2}], and according to either of the two if $f_1=0$.
\end{ABC}
\end{Thm}
\begin{Rem}\rm\small \label{remmain}
  \begin{ABC}
    \item This result of course also covers the ideal model ($r=0$), and is also relevant for the fixed neighborhood approach:
    If for fixed $n$, we formally plug in $r=s\sqrt{n}$ (for $s$ small in comparison to $\sqrt{n}\,$) this gives a corresponding result
    for the maximal MSE of the sample median on a neighborhood of fixed size $s$. (``formal'', as we cannot control the remainder for arbitrary $s<1$.)
    \item \label{n1} If one is only interested in the behavior of $n\, {\rm MSE}$ up to order $\Lo(n^{-1/2})$, one may weaken assumption
     \eqref{tayl1} to: For      some $\delta>0$,
    \begin{equation}\label{atayl1}
f(x)=f_0+f_1x+\LO(x^{1+\delta}),\qquad f_0>0
    \end{equation}
    \item\label{rema} Conditions \eqref{contbed1} and \eqref{contbed2} imply that it is sufficient to contaminate $F^n$ by
    measures $H_{n}$ the one dimensional marginals of which are either concentrated right of $C\,\sqrt{\log(n)/n}$
    or left of $-C\,\sqrt{\log(n)/n}$ for some constant $C>0$ in order
    to obtain a maximal MSE --- an astonishingly modest contamination!  With respect to\ \eqref{ALEd}, this is plausible however, as $|\psi_{\SSs \rm Med}|$ attains its maximal value for  any $x\not=0$.\\
    The thinning out of the marginals by means of Defintion~\ref{thinout} even tells us that of
    the $n$ potentially contaminating $X_i^{\SSs \rm di}$ only all subsets of cardinality roughly $\sqrt{n}$
    need to be ``large'' at all,  the remaining coset (of cardinality order $n(1+\Lo(1))$) of contaminations
     might even stem from the ideal situation!\\
     As shown in Proposition~\ref{neccond}, conditions~\eqref{contbed1} resp.\ \eqref{contbed2} are almost necessary.

    \item The sample median for odd sample size as well as all variants of the median
    considered in Proposition~\ref{medeventhm} come up with the same leading term $
    {(1+r^2)}/{(4f_0^2)}$ for $n\,{\rm MSE}$--- according to first order asymptotics \eqref{modifloss} (with modified loss there!).
    \item In all variants of the sample median considered in Theorem~\ref{medthm} and Proposition~\ref{medeventhm},
    the second order correction is positive, so that for any $r>0$ we eventually underestimate the MSE by first order asymptotics.
  \end{ABC}
\end{Rem}
\subsection{Ramifications}\label{ramifsec}
As simulations in section~\ref{simstudy} were made for even sample size, we present an analogue to
Theorem~\ref{medthm} for even sample size below. As there are infinitely many sample medians for even sample size,
we consider the following variants:
\begin{itemize}
    \item the order statistics $X_{[m:n]}$
    \item the order statistics $X_{[(m+1):n]}$
    \item the randomized estimator $M'_n:=UX_{[m:n]}+(1-U)X_{[(m+1):n]}$
          with some randomization $U\sim{\rm Bin}(1,1/2)$
    \item the midpoint--estimator $\bar M_n:=(X_{[m:n]}+X_{[(m+1):n]})/2$
    \item the bias corrected estimator $M''_n:=(X_{[m:n]}+\frac{1}{2n\,f_0})$
\end{itemize}
\begin{Prop}\label{medeventhm}
Under the assumptions of Theorem~\ref{medthm}, for even sample size $n=2m$,
for the sample median variants $X_{[m:n]}$, $X_{[(m+1):n]}$, $M_n'$, $\bar M_n$, $M''_n$,
here denoted by $M_n$ generically, for any $\ve<1/2$, for $G^{(n)}$ varying in $\tilde {\cal Q}_n(r,\ve)$
of \eqref{contdef} it holds
\begin{eqnarray}
&&  \sup_{G^{(n)}}\,n\,[{\rm MSE}(M_n,G^{(n)})]=\frac{1}{4 f_0^2}\Bigg\{
(1+r^2)+\Tfrac{r}{\sqrt{n}}\Big(a_{1,0}+a_{1,1}\Tfrac{f_1}{f_0^2}\Big)+\nonumber\\
&&\qquad +\Tfrac{1}{n}\Big(a_{2,0}+a_{2,1}\Tfrac{f_1}{f_0^2}+a_{2,2}\Tfrac{f_2}{f_0^3}+
a_{2,3}\Tfrac{f_1^2}{f_0^4}\Big)\Bigg\}+\Lo(\Tfrac{1}{n})\label{Malleven}
\end{eqnarray}
for some real numbers $a_{i,j}=a_{i,j}(M_n)$ which are given in detail in Proposition~\ref{aijterms}.\\
In any variant, the maximal contamination is achieved by contaminating measures $H_{n}$
according to either condition \eqref{contbed1} or \eqref{contbed2} where the distinction between
these two is made as in the case of odd sample size.
\end{Prop}
%
\begin{Prop} \label{aijspec}
[{Specification of the terms \boldmath$a_{i,j}$}] \label{aijterms}$\mbox{ }$
%
Splitting up $a_{2,0}$, $a_{2,1}$, $a_{2,2}$ according to
\begin{equation}\label{erstspec}
a_{2,0}=a_{2,0,r}+a_{2,0,c},\qquad a_{2,1}=a_{2,1,r}+a_{2,1,c},\qquad a_{2,2}=a_{2,2,r}+a_{2,2,c}
\end{equation}
we get
\begin{ABC}
\item {\em Identical terms for all variants\/}:
\begin{equation}
a_{2,3}  =\Tfrac{5(r^4+6r^2+3)}{16}, \qquad 
a_{2,2,c}=-1/4, \qquad a_{2,2,r}=-\Tfrac{(r^4+6r^2)}{12}
\end{equation}
\item {\em Varying terms in the ideal model\/}:  
\begin{equation}
\begin{array}{l}
a_{2,0,c}(M''_n)=-2,\qquad a_{2,0,c}(\bar M_{n})= -3\\
a_{2,0,c}(X_{[m:n]})= a_{2,0,c}(X_{[(m+1):n]})= a_{2,0,c}(M'_n) = -1
\end{array}
\end{equation}
\item {\em Remaining $a_{i,j}$ for $\bar M_n$, $M'_n$, and $M''_n$\/}:
\begin{equation}
\begin{array}{l}
a_{1,0}(M''_n)=a_{1,0}(M'_n)=a_{1,0}(\bar M_n)=2(1+r^2),\\
a_{1,1}(M''_n)=a_{1,1}(M'_n)=a_{1,1}(\bar M_n)=(r^2+3)\sign(f_1)/2,
\end{array}
\end{equation}
\begin{equation}
\begin{array}{l}
a_{2,0,r}(M'_n)=a_{2,0,r}(\bar M_n)=3r^4+3r^2=a_{2,0,r}(M''_n)-2r^2\sign(f_1)
\end{array}
\end{equation}
\begin{equation}\label{letztspec}
\begin{array}{l}
a_{2,1,c}(M'_n)=a_{2,1,c}(\bar M_n)=0,\quad a_{2,1,c}(M''_n)=1,\\
a_{2,1,r}(M'_n)=a_{2,1,r}(\bar M_n)=\Tfrac{3\,r^2(3+r^2)\sign(f_1)}{2}=a_{2,1,r}(M''_n)-r^2
\end{array}
\end{equation}
\item {\em Remaining $a_{i,j}$ for $X_{[m:n]}$ and $X_{[(m+1):n]}$}:
\begin{equation}
a_{2,1,c}(X_{[m:n]})=3/2=-a_{2,1,c}(X_{[m:n]})
\end{equation}
For $X_{[m:n]}$ and $X_{[(m+1):n]}$, condition~\eqref{contbed1} [\eqref{contbed2}] applies 
if $4f_0^2 >\;[<]\; -(3+r^2)f_1$ 
Correspondingly, let \begin{equation}
s'=\left\{\begin{array}{lll}
  \hphantom{-}1&\mbox{for}&X_{[m:n]}\\
  -1&\mbox{for}&X_{[(m+1):n]}
\end{array}\right.
\end{equation}
and
\begin{equation}
s=\sign((3+r^2)f_1+s' 4f_0^2 )\label{s-def}
\end{equation}
Then the remaining $a_{i,j}$ for  $X_{[m:n]}$ and $X_{[(m+1):n]}$ are given by
\begin{equation}
\begin{array}{rclrcl}
a_{1,0}&=&2+2s's+2r^2, & a_{2,1,r}&=&3s's\Big(\big(3+s\big)r^2+r^4\Big)/2\\
a_{2,0,r}&=&3r^4+\big(3+4s\big)r^2
\end{array}
\end{equation}
In case $4f_0^2 = s'(3+r^2)f_1$,  both condition~\eqref{contbed1} and \eqref{contbed2} up to $\Lo(n^{-2})$ lead to the same MSE.
\end{ABC}
\end{Prop}
\begin{Rem}\rm\small \label{oddnot}
In case of the sample median for odd sample size,
\begin{equation}
  \begin{array}{rclrclrcl}
    a_{1,0}&=&2(1+r^2),&a_{1,1}&=&\frac{(r^2+3)\sign(f_1)}{2}&a_{2,0,c}&=&-2,\\
    a_{2,0,r}&=&3r^2+3r^4,&a_{2,1,c}&=&0,&a_{2,1,r}&=&\frac{3r^2(3+r^2)\sign(f_1)}{2},\\
    a_{2,2,c}&=&-\frac{1}{4},&a_{2,2,r}&=&-\frac{6r^2+r^4}{12},&a_{2,3}&=&\frac{5(3+6r^2+r^4)}{16}
      \end{array}\nonumber
\end{equation}
\end{Rem}
%
\begin{Rem}\rm\small
It is a well-known consequence of the Jensen inequality that convexity of both loss and admitted estimation (or more generally decision)
domain entails that randomization cannot improve an averaged estimator, compare e.g.\ \citet[(1.2.98), p.~52]{Wi:85}. This is reflected
by the fact that in both ideal and contaminated situation, $\bar M_n$ up to $\Lo(1/n^2)$ has a smaller ${\rm MSE}$ than $M'_n$---
the only difference arising in term $a_{2,0,c}$.
\end{Rem}
\begin{Rem}\rm\small
In the ideal model, as shown in \citet[Theorem~1]{C:M:S:94}, one even has the peculiarity that, in our notation
\begin{equation}\label{deter}
{\rm MSE}(\bar M_{2m},F)-{\rm MSE}(M_{2m+1},F)=-\frac{1}{16m^3 f^2_0}+\Lo(m^{-3})
\end{equation}
that is, evaluating the sample median at one more observation (from $2m$ to $2m+1$) deteriorates ${\rm MSE}$!
As our expansion already stops at $\Lo(m^{-2})$, we cannot reproduce \eqref{deter} to the given
exactitude by means of our representations \eqref{Mall} and  \eqref{Malleven}.\\
After correcting (minor) typing errors in formulae (2.2), (2.5), and (2.6) in the cited reference, we
obtain \eqref{Mall} and  \eqref{Malleven} from (2.2) again; for details refer to 
the web-page to this article.
\end{Rem}
Conditions~\eqref{contbed1} / \eqref{contbed2} almost characterize the risk-maximizing contaminations:
\begin{Prop} \label{neccond}
Under the assumptions of Theorem~\ref{medthm}, let $\delta_0$. Assume that, for $K=\sum_{i=1}^n U_i$ and $k>(1-\delta)r\sqrt{n}$,
\begin{equation} \label{condnec}
\Pr\Big(\sum_{i=1}^n U_i \Jc(X_i^{\rm \SSs di} \le \sqrt{\log(n)/n}/(2f_0))\geq 1\,\Big|K=k\Big)\ge p_0>0
\end{equation}
Then, eventually in $n$, no such sequence of contaminations $G_\flat^{(n)}\in\tilde{\cal Q}(r)$, can attain the maximal MSE in \eqref{Mall} as in condition~\eqref{contbed2}
(i.e.\ with positive bias).
More precisely,
\begin{equation}
\sup_{G^{(n)}}\,n\,[{\rm MSE}(M_n,G^{(n)})]-\,n\,[{\rm MSE}(M_n,G_\flat^{(n)})] \ge
\frac{p_0}{2nf_0\sqrt{2\pi } }+\Lo(1/n)
\end{equation}
A corresponding relation holds for condition~\eqref{contbed1}.
\end{Prop}
With the same techniques we can also specify which parts of the MSE ---up to order $1/n^2$--- are due to variance and which are due to
bias; to this end let $M_n$ be the sample median and the midpoint estimator $\bar M_n$ for odd resp.\ 
even sample size.
\begin{Prop}\label{BiasVar}
In the situation of Theorem~\ref{medthm}, for contaminating measures $H_{n}$ as spelt
out in \eqref{contbed1}, \eqref{contbed2}, leading to $G^{(n)}_0$
in \eqref{contdef}, 
it holds
\begin{eqnarray}
\! \! &\! &    \!   n\,[{\rm Var}(M_n,G^{(n)}_0)]\;=\;\frac{1}{4 f_0^2}\,\Bigg\{\,1\,+\,\Tfrac{r}{\sqrt{n}}\Big({\Ss 2}+{\Ss |f_1|}\Big)+\nonumber\\
\! \! &\! &\!\qquad+\Tfrac{1}{n}\Big({\Ss 3r^2-(5-(-1)^n)/2}+\Tfrac{3|f_1|r^2}{f_0^2}-\Tfrac{f_2(r^2+1)}{4f_0^3}
+\Tfrac{f_1^2(8r^2+7)}{8f_0^4}
\Big)\Bigg\}+\Lo(\Tfrac{1}{n})
\\
\! \! &\! &    \! \sqrt{n}\,\,\Big|{\rm Bias}(M_n,G^{(n)}_0)\,\Big|\;=\;\frac{1}{2 f_0}\,\Bigg\{\,r\,+\,\Tfrac{1}{\sqrt{n}}\Big({\Ss r^2}-
\Tfrac{|f_1|(r^2+1)}{4f_0^2}\Big)+
\nonumber\\
\! \! &\! &\!\qquad+\Tfrac{r}{n}\Big({\Ss r^2}-\Tfrac{|f_1| \, (r^2+1)}{2f_0^2}+\Tfrac{f_2(r^2+3)}{24f_0^3}
+\Tfrac{f_1^2(r^2+3)}{8f_0^4}
\Big)\Bigg\}+\Lo(\Tfrac{1}{n})
\end{eqnarray}
\begin{eqnarray}
\! \! &\! &    \!  n\,[{\rm Bias}^2(M_n,G^{(n)}_0)]\;=\;\frac{1}{4 f_0^2}\,\Bigg\{\,r^2\,+\,\Tfrac{r}{\sqrt{n}}\Big({\Ss 2r^2}+\Tfrac{|f_1|(r^2+1)}{2}\Big)+
\nonumber\\
\! \! &\! &\!\qquad+\Tfrac{1}{n}\Big({\Ss 3r^4}+\Tfrac{3|f_1|r^2(r^2+1)}{2f_0^2}-\Tfrac{f_2r^2(r^2+3)}{12f_0^3}
+\Tfrac{f_1^2(5r^4+14r^2+1)}{16f_0^4}
\Big)\Bigg\}+\Lo(\Tfrac{1}{n})
\end{eqnarray}
\end{Prop}
We next specialize Theorem~\ref{medthm} and Proposition~\ref{medeventhm} for the case 
of $F={\cal N}(0,1)$ for later comparison to numeric and simulated values.
%
\begin{Cor}\label{medthmnormal}
In the location model about $F={\cal N}(0,1)$,
\begin{eqnarray}
\!\sup_{G^{(n)}}\,n\,[{\rm MSE}(M_n,G^{(n)})]=
\frac{\pi}{2}\Bigg\{
(1+r^2)+\Tfrac{r}{\sqrt{n}}a_{1,0}+\Tfrac{1}{n}\Big(a_{2,0}+2\pi\,a_{2,2}\Big)\Bigg\}+\Lo(\Tfrac{1}{n})
\end{eqnarray}
\end{Cor}%
\begin{Cor}\label{medthmnormid}
In the location model about $F={\cal N}(0,1)$, in the ideal model
\begin{equation}\label{mednormideq}
n\,{\rm MSE}({\rm Med}_n,F)= \frac{\pi}{2}[1+(\frac{\pi}{2}+a_{2,0,c})/n]+\Lo(\Tfrac{1}{n}),
\end{equation}
As numerical evaluation of \eqref{mednormideq}, we get in the three cases:
\begin{equation}\label{aseq1}
n\,{\rm MSE}({\rm Med}_n,F)\doteq \Lo(\Tfrac{1}{n})+\left\{\!\begin{array}{lcl}
1.5708(1-0.4292/n)&\!&\!{\Ss \mbox{\scriptsize for }{\rm Med}_n,\,M_n''}\\
1.5708(1+0.5708/n)&\!&\!{\Ss \mbox{\scriptsize for }X_{[m:n]},\, X_{[(m+1):n]},\, M'_n}\\
1.5708(1-1.4292/n)&\!&\!{\Ss \mbox{\scriptsize for }\bar M_n}
\end{array}\right.
\end{equation}
\end{Cor}%
This means: We overestimate ${\rm MSE}({\rm Med}_n,F)$ by the first order asymptotics
for odd sample size $n$ and with estimator $M''_n$, and to an even
higher degree, if we use $\bar M_n$. The risk of estimators $X_{[m:n]}$, $X_{[(m+1):n]}$, $M'_n$ however
is  underestimated.
\section{Illustration of the results}\label{illustrsec}\req
To illustrate the approximation, we consider the case of $F={\cal N}(0,1)$ with a number of
numerical evaluations and a small simulation study.
\subsection{Numerical Results in the ideal model}\label{numres}
In the ideal model, we have evaluated the integrals numerically, using formulas for the densities in the ideal model
to be derived later in section~\ref{proofsec}: $g_{n}$ for
 the sample median for odd sample size from \eqref{mediddens} and $g_{n}$ for the midpoint estimator for even sample size
from \eqref{denseevenid}.  For the numerical
calculations, we have used {\tt R 2.11.0}. Note that
the limit up to five digits in this case is $1.5708$.
%
Further sample sizes are available on 
the web-page to this article.
\begin{table}
\small
\begin{center}
\begin{tabular}{r||r|r|r|r|r}
                                                 &
\multicolumn{1}{c|}{num.}                          &
\multicolumn{4}{c}{error of asymptotics} \\[0.2ex]
$n$&\multicolumn{1}{c|}{exact}&\multicolumn{2}{c|}{1st/2nd order} &
\multicolumn{2}{c}{3rd order}\\
&\multicolumn{1}{c|}{${\rm Var}^{\rm\SSs id}_n$}& \multicolumn{1}{c|}{absolute} & \multicolumn{1}{c|}{relative} & \multicolumn{1}{c|}{absolute} & \multicolumn{1}{c}{relative} \\
\hline 
\multicolumn{6}{c}{$\boldmath M_n$}\\[0.5ex]
$5    $     &  $    1.4341$   &    $  1.366\,{\rm \Ss E}-1 $    &  $  9.527\,\% $  &    $    1.790\,{\rm \Ss E}-3$     & $ 0.125\,\% $\\
$11   $     &  $    1.5088$   &    $  6.201\,{\rm \Ss E}-2 $    &  $  4.110\,\% $  &    $    7.194\,{\rm \Ss E}-4$     & $ 0.048\,\% $\\
$101  $     &  $    1.5641$   &    $  6.687\,{\rm \Ss E}-3 $    &  $  0.428\,\% $  &    $    1.174\,{\rm \Ss E}-5$     & $ 0.001\,\% $\\[0.5ex]
%
\multicolumn{6}{c}{$\boldmath X_{[n/2:n]}$, $\boldmath X_{[(n/2+1):n]}$, $\boldmath M'_n$}\\[0.5ex]
%
$6      $     &  $    1.7210$   &    $  -1.502\,{\rm \Ss E}-1 $    &  $-8.728\,\%$  &    $-7.715\,{\rm \Ss E}-4$     & $-0.044\,\% $\\
$10     $     &  $    1.6610$   &    $  -9.022\,{\rm \Ss E}-2 $    &  $-5.431\,\%$  &    $-5.560\,{\rm \Ss E}-4$     & $-0.033\,\% $\\
$100    $     &  $    1.5798$   &    $  -8.976\,{\rm \Ss E}-3 $    &  $-0.568\,\%$  &    $-9.445\,{\rm \Ss E}-6$     & $-0.001\,\% $\\[0.5ex]
%
%
\multicolumn{6}{c}{$\boldmath M''_n$}\\[0.5ex]
$6     $    &  $    1.4776$   
&    $  9.320\,{\rm \Ss E}-2 \!$    &  $ 6.307\,\% $  &    $    -1.917\,{\rm \Ss E}-2$     & $ -1.297\,\%$\\
$10     $    &  $    1.5106$   
&    $  6.019\,{\rm \Ss E}-2 \!$    &  $ 3.984\,\% $  &    $    -7.233\,{\rm \Ss E}-3$     & $ -0.479\,\% $\\
$100    $    &  $    1.5641$   
&    $  6.665\,{\rm \Ss E}-3 \!$    &  $ 0.426\,\% $  &    $    -7.681\,{\rm \Ss E}-5$     & $ -0.005\,\% $\\[0.5ex]
%
\multicolumn{6}{c}{$\boldmath \bar M_n$}\\[0.5ex]
$6      $    &  $    1.2884$   & $-2.823\,{\rm \Ss E}-1   $&    $-21.913\,\% $    & $-9.182\,{\rm \Ss E}-2 $  &  $ -7.126\,\% $\\
$10     $    &  $    1.3832$   & $-1.875\,{\rm \Ss E}-1   $&    $-13.557\,\% $    & $-3.697\,{\rm \Ss E}-2 $  &  $ -2.672\,\% $\\
$100    $    &  $    1.5488$   & $-2.200\,{\rm \Ss E}-2   $&    $ -1.421\,\% $    & $-4.472\,{\rm \Ss E}-4 $  &  $ -0.029\,\% $\\
\end{tabular}
\caption{\label{tabel1} Accuracy of the asymptotics in the ideal model}%
\end{center}
\end{table}
\subsection{A simulation study}\label{simstudy}
\subsubsection{Simulation design}
Under {\tt R 2.11.0}, compare \citet{RMANUAL}, we simulated $M=10000$ runs of sample size
$n=5,10,30,100$
in the ideal location model ${\cal P}={\cal N}(\theta,1)$ at $\theta=0$.
In a contaminated situation, we used observations stemming from
$$
G_s^{(n)}={\cal L}\{[(1-U_i)X_i^{\SSs \rm id}+U_i X_i^{\SSs \rm cont}]_i\,\Big|\,\sum U_i\leq \ulcorner n/2 \urcorner -1\,\}
$$
for $U_i\iid {\rm Bin}(1,r/\sqrt{n})$, $X_i^{\SSs \rm id}\iid{\cal N}(0,1)$, $X_i^{\SSs \rm cont}\iid \Jc_{\{100\}}$
all stochastically independent and for contamination radii
$r=0.1, 0.5, 1.0$. Further results for $n=30,50$ and/or $r=0.25,0.5$ are available on 
the web-page to this article.
With respect to Remark~\ref{remmain}~\eqref{rema},
the contamination point $100$ will largely suffice to attain the maximal MSE on $\tilde {\cal Q}_n$.
\subsubsection{Results}
The simulated results for $n\,{\rm MSE}({\rm Med}_n, G^{(n)}_s)$ come with an asymptotic $95\%$--confidence interval,
which is based on the CLT for the variable \begin{equation}
\overline{\rm empMSE}_n=\Tfrac{n}{10000}\sum\nolimits_j [{\rm Med}_n({\rm sample}_j)]^2
\end{equation}
We compare these results to the
corresponding numerical ``exact'' values and to the asymptotical values for approximation order
$n^0$, $n^{-1/2}$, $n^{-1}$ respectively. For even $n$ we take the midpoint--estimator which is the
default procedure in {\tt R}. For the numerical evaluations we use density formulas from section~\ref{proofsec}: $g_{n,k,k}$ for
  odd sample size from \eqref{denseodd} and the integrand from \eqref{denseeven} for even sample size.\\
{\footnotesize  For the ideal situation we had simulation results available for all runs to $r\not=0$,
so the actual sample size for $r=0$ is $40000$.}
\smallskip\\
\begin{table}
\small
\begin{center}
\begin{tabular}{l|r||r@{\hspace{0.5em}\hfill$[$}c@{;}c@{$\,]\;$}|r|rrr}
  \multicolumn{1}{c|}{n}&  \multicolumn{1}{c||}{r}&  \multicolumn{1}{c}{sim}& \multicolumn{1}{c}{[low;}&  \multicolumn{1}{c|}{up]}&\multicolumn{1}{c}{num}&  \multicolumn{1}{c}{$n^0$}&  \multicolumn{1}{c}{$n^{-1/2}$}&
    \multicolumn{1}{c}{$n^{-1}$}\\ \hline
$n=5$& $0.00$&  $1.423$& $1.384$& $1.464$& $ 1.434$&  $ 1.571$&  $ 1.571$&  $ 1.436$\\
     & $0.10$&  $1.652$& $1.602$& $1.701$& $ 1.671$&  $ 1.587$&  $ 1.728$&  $ 1.613$\\
     & $0.50$&  $3.014$& $2.917$& $3.111$& $ 3.045$&  $ 1.963$&  $ 2.842$&  $ 3.258$\\
     & $1.00$&  $4.525$& $4.394$& $4.655$& $ 4.509$&  $ 3.142$&  $ 5.952$&  $ 8.853$\\
\hline
$n=10$&  $0.00$&  $  1.371$&  $1.333$& $1.410$& $ 1.383$&  $ 1.571$&  $ 1.571$&  $ 1.346$\\
      &  $0.10$&  $  1.534$&  $1.491$& $1.578$& $ 1.521$&  $ 1.587$&  $ 1.687$&  $ 1.472$\\
      &  $0.50$&  $  2.980$&  $2.882$& $3.078$& $ 2.916$&  $ 1.963$&  $ 2.584$&  $ 2.636$\\
      &  $1.00$&  $  5.723$&  $5.568$& $5.879$& $ 5.735$&  $ 3.142$&  $ 5.129$&  $ 6.422$\\
      \hline
$n=30$&  $0.00$&  $  1.518$&  $1.476$& $1.560$& $ 1.501$&  $ 1.571$&  $ 1.571$&  $ 1.496$\\
      &  $0.10$&  $  1.614$&  $1.569$& $1.659$& $ 1.579$&  $ 1.587$&  $ 1.644$&  $ 1.573$\\
      &  $0.50$&  $  2.400$&  $2.331$& $2.469$& $ 2.390$&  $ 1.963$&  $ 2.322$&  $ 2.339$\\
      &  $1.00$&  $  5.391$&  $5.245$& $5.538$& $ 5.255$&  $ 3.142$&  $ 4.289$&  $ 4.720$\\
       \hline
$n=100$&  $0.00$&  $  1.546$&  $1.503$& $1.589$& $ 1.549$&  $ 1.571$&  $ 1.571$&  $ 1.548$\\
      &   $0.10$&  $  1.585$&  $1.541$& $1.629$& $ 1.597$&  $ 1.587$&  $ 1.618$&  $ 1.597$\\
      &   $0.50$&  $  2.165$&  $2.106$& $2.223$& $ 2.171$&  $ 1.963$&  $ 2.160$&  $ 2.165$\\
      &   $1.00$&  $  4.010$&  $3.911$& $4.108$& $ 3.952$&  $ 3.142$&  $ 3.770$&  $ 3.899$
\end{tabular}
\caption{\label{tabel2} Asymptotics compared to numerical and simulational evaluations}%
\end{center}
\end{table}
\begin{figure}
  \begin{center}
    \includegraphics[width=13cm,height=11.5cm]{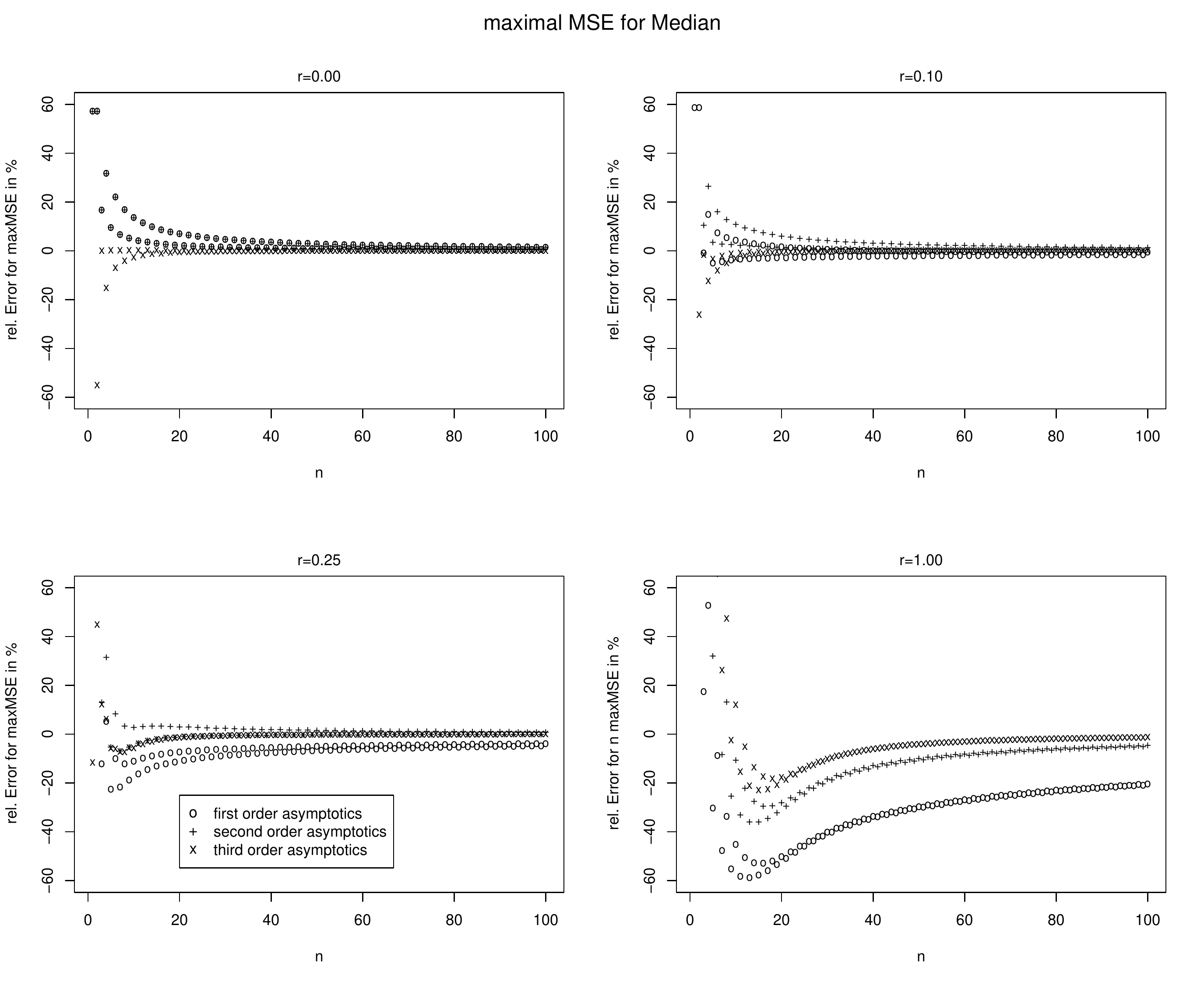}
    \caption{\label{fig1}{\small The mapping $n\mapsto {\rm rel.error}({\rm maxMSE}({\rm Med}_n))$
    for $F={\cal N}(0,1)$.}}
  \end{center}
\end{figure}
%
\begin{table}
\begin{center}
\small
\begin{tabular}{l|l||r|r|r|r|r}
${\rm rel.err}$& order&  $r=0.00$& $r=0.10$& $r=0.25$& $r=0.50$& $r=1.00$\\
\hline
1\%&1st order asy.& 143&  320&    2449&  10016& 40127\\
&2nd order asy.& 143&  133&      85&    124& 479\\
&3rd order asy.&  17&   17&      25&     48 &124\\
\hline
5\%&1st order asy.&   29&   9&    92&     406&    1629\\
&2nd order asy.&   29&  25&    10&      30&    101\\
&3rd order asy.&    7&   9&    11&      20&    46\\
\end{tabular}
\caption{\label{tabel2n} \small Minimal $n_0$ s.t.\ for $n\ge n_0$ the relative error using first to third order asymptotics
for approximating ${\rm maxMSE}({\rm Med}_n)$ on $\tilde {\cal Q}_n(r,\ve)$ is smaller than $1\%$ resp.\ $5\%$}%
\end{center}
\end{table}
\subsection{Discussion}\label{disksimstudy}
The numerical results of subsection~\ref{numres} show an excellent approximation quality of our formulas \eqref{Mall}
and \eqref{Malleven} with specifications \eqref{erstspec} to \eqref{letztspec} in the ideal model. In particular
the different under/over--estimation properties of the different median variants are closely reflected by the
numerical results.  The approximation quality of the midpoint estimator indicated in \eqref{aseq1}
is somewhat less well supported by the numerical results, which is probably due to the fact, that by iterated numerical integration the accuracy
of the numerical approximation will be inferior to the other variants.\\
In the contaminated situation, empirical and numerical results also strongly support our assertion of a good
approximation quality down to moderate to very small sample sizes, as long as the contamination radius $r$ is not too large:
For $n=5$ upto radius $r=0.1$, for $n=10$ (almost) upto $r=0.25$, for $n>30$ upto $r=0.5$, all approximations up to $\Lo(n^{-1})$--terms
stay within an (empirical) $95\%$--confidence interval around the (empirical) MSE (multiplied by $n$).\\
In any case, higher order asymptotics yield more accurate approximations than first order ones, and upto case $n=5$, the $1/n$--terms
improve the approximation with respect to the $1/n^{1/2}$--terms.\\
A closer look is provided by figure~\ref{fig1} (and, zooming in for $n\ge 16$, there is an additional figure on the web-page). Indeed for all investigated radii $r=0, 0.10, 0.25, 1.00$,
the relative error of our asymptotic formula w.r.t.\ the corresponding numeric figures is quickly decreasing in absolute value in $n$;
also, we notice  a certain oscillation between odd and even sample sizes induced by the different definitions of the sample median
in these cases. In table~\ref{tabel2n}, we have determined the smallest sample size $n_0$ such that for $n\ge n_0$ the relative error using first to third order asymptotics
for approximating ${\rm maxMSE}({\rm Med}_n)$ on $\tilde {\cal Q}_n(r)$ is smaller than $1\%$ resp.\ $5\%$ which shows that for
$r\leq 0.5$ we need no more than $20$ $(50)$ observations to stay within an error corridor of $5\%$ ($1\%$) in third order asymptotics.
For first order asymptotics, however  we need considerable sample sizes for reasonable approximations unless the radius is rather small.
\appendix
\section{Proofs}\label{proofsec}\req
\subsection{Proof of Remark~\ref{slowdecay}(a)}\label{slowdecayp}
Let $n=2m+1$ and $\gamma\in(0,1)$. Necessity: With $\bar F=1-F$, by integration by parts and H\"older inequality to exponent $m+1$,
we obtain that for any $T>0$, and some constants $K',K>0$ and $\alpha=\frac{\gamma-1}{m+1}$
\begin{eqnarray*}
&&\!\Ew_F|{\rm Med}_n|^\gamma=n { 2m\choose m} \int|x|^\gamma \bar F(x)^m F(x)^m\,F(dx)\ge \\
&\!\ge\!&\! K\max\Big(\int_T^\infty \!x^{\gamma-1} F(-x)^{m+1}\, dx, \int_T^\infty \!x^{\gamma-1} \bar F(x)^{m+1}\, dx\Big)\ge
K' \Big(\int_{\SSs \{|x|>T\}}\! \!|x|^{\alpha}\, F(dx) \Big)^{m+1}
\end{eqnarray*}
Sufficiency: Under condition~\eqref{abkling}, for $g_\delta(t)=|t|^\delta F(t)\bar F(t)$ and $\hat g_\delta:=\sup_t g_\delta(t)$
it holds ---cf.\  \citet[(2.37)]{Jur:82}
\begin{equation} \label{Jurec}
\hat g_\delta< \infty,\quad \lim_{|t|\to\infty} g_\delta(t)=0,\quad I_b:=\int [F(t)\bar F(t)]^b\,dt <\infty\;\;\forall b\ge 1/\delta
\end{equation}
Hence for any $n>1+2\gamma/\delta$, it follows $b=m+(1-\gamma)/\delta>1/\delta$ and hence
\begin{eqnarray*}
\!&&\!\Ew_F|{\rm Med}_n|^\gamma= n\gamma  { 2m\choose m} \int_0^\infty x^{\gamma-1} [\bar F^{m+1}(x) F^m(x)+\bar F^{m}(-x) F^{m+1}(-x)]\,dx\le \\
\!&\!\leq \!& \!n\gamma  { 2m\choose m}  \hat g_\delta^{(\gamma-1)/\delta} \int_{-\infty}^\infty [\bar F(x) F(x)]^{b} \,dx\leq
n\gamma  { 2m\choose m}  \hat g_\delta^{(\gamma-1)/\delta} I_{b} <\infty
\end{eqnarray*}
The arguments for even sample size are similar.
\hfill\qed
\subsection{Proof of Proposition~\ref{badmed}}
The assertion for uniform normality follows along the lines of  \citet[Theorem~6.2.8]{Ri:94}:
Although the assumed uniform Lipschitz continuity of the scores $\psi$---(68), p.~231 in the cited reference---fails,
a look into the proof of the theorem shows that this condition only is needed to achieve
 conclusion $dL(\theta)=\EM_k$ on p.~235, which in our  situation is the case anyway.\\
Assertion \eqref{>C} is shown by a breakdown-point argument:
We interpret $G^{(n)}$ according to \eqref{Uidef}, where for this proof $X_i^{\SSs \rm di}\iid \Jc_{\{x_0\}}$.
We observe that ${\rm Med}_n\geq x_0$ surely under $G^{(n)}$ as soon as  $K=\sum U_i$, the number of observations stemming from $\Jc_{\{x_0\}}$,
is larger than $m$. But, $K$ being a binomial variable, the event $\{K>m\}$ carries positive probability $p_n$.
So setting
$x_0:=\sqrt{
{C}/{p_n}\,}$,
we get\\
$
{\rm MSE}({\rm Med}_{n},{G^{(n)}_{0}})=\Ew_{G^{(n)}_{0}}({\rm Med}_n^2)\geq
\Ew_{G^{(n)}_{0}}({\rm Med}_n^2\Jc_{\{K>m\}})\geq x_0^2 p_n=C
$
\hfill\qed

\subsection{Outline of the proof of Theorem~\ref{medthm}}
As in the theorem we define $n=2m+1$ and first consider the situation knowing
that exactly $K=\sum U_i=k$ observations have been contaminated,
to values say $\tilde x_1,\ldots, \tilde x_k$. More specifically, it will
 be sufficient to consider---for each fixed $t$---the number
\begin{equation}
  j=j_k(t):=\#\{\tilde x_i: \tilde x_i\geq t\}
\end{equation}
In this situation we will derive the (conditional) probability that the (unique) median ${\rm Med}_n$
is not larger than $t$ and derive its density. We then fix some $k_1>1$ and $k_2>\sqrt{5/2}$ and
split up the proof according to the following tableau
\begin{center}
\begin{tabular}{c||c|c|c}
&$K\leq k_1  r\sqrt{n}$ &  $k_1r\sqrt{n}< K\leq \rho\,n $ & $K> \rho\, n$\\
\hline\hline
$|t|< k_2\sqrt{\log(n)/n}/f_0$& (I) &&\\
\cline{1-2}
$k_2\sqrt{\log(n)/n}/f_0 \leq |t|\leq n^2 $& (II)& \raisebox{1.5ex}[-1.5ex]{(III)}& {\footnotesize excluded} \\
\cline{1-3}
$|t|\geq n^2$&  \multicolumn{2}{c|}{(IV)}& \\
\hline
\end{tabular}
\end{center}
%

%
For cases (II) to (IV), we will show that they contribute only terms of order $\Lo(n^{-1})$ to $n\,{\rm MSE}({\rm Med}_n)$
and hence can be neglected. Applying Taylor expansions at large, we derive an expression in which it becomes clear, that independently from $t$
and eventually in $n$, the maximal MSE is attained for $j_k(t)$ either identically $k$ or identically $0$ for
all $t$ in (I)---or equivalently all $\tilde x_i$ are either smaller than $-\Tfrac{k_2}{f_0}\sqrt{\log(n)/n}$
 or larger than $\Tfrac{k_2}{f_0}\sqrt{\log(n)/n}$.
Integrating out first $t$ and then $k$ we obtain the result.
\subsection[Law of the Median in ideal and contaminated situation]{\boldmath${\cal L}({\rm Med}_n)$ in ideal and contaminated situation}
\subsubsection{Ideal Situation}
\begin{Lem}
Let $X_i\iid P$ real-valued random variables.
Then
\begin{equation}\label{VertMede}
P(X_{[k:n]}\leq t)=\sum_{l=k} {{n}\choose{l}}  P(t)^l (1-P(t))^{n-l}
\end{equation}
If $dP=p\,\,d\lambda$, then $X_{[k:n]}$ has density
\begin{equation}\label{Xkiddens}
  g(t)=n p(t) {{n-1}\choose{k-1}}  P(t)^{k-1} (1-P(t))^{n-k}
\end{equation}
In particular the density of the sample median for odd sample size $n=2m+1$
in the situation of Theorem~\ref{medthm} is
\begin{equation}\label{mediddens}
  g_n(t)=n f(t) {{2m}\choose{m}}  F(t)^{m} (1-F(t))^{m}
\end{equation}
\end{Lem}
\begin{proof}{}
The proof is standard, but as we will need some terms later, we pass through the main steps here:
For fixed $t\in\R$ we introduce $Y_i:=\Jc_{\{X_i \leq t\}}$. Then
the following events are identical
\begin{equation}\label{eqevent}
\{X_{[k:n]}\leq t\}=\{ \# i:\{X_i \leq t\}\geq k\}= \{ \sum_{i=1}^n Y_i\geq k\}
\end{equation}
The fact that $Y_i\iid{\rm Bin}(1,P(t))$ entails \eqref{VertMede}.
\eqref{Xkiddens} follows by simple differentiating, \eqref{mediddens} by plugging in $k=m+1$.
\qed\end{proof}
\subsubsection{Contaminated situation}
By \eqref{contdef}, $X_i=(1-U_i)X^{\SSs \rm id}_i+U_iX^{\SSs \rm di}_i$, and thus fixing again $t\in\R$, also
\begin{equation}
  Y_i=(1-U_i)Y^{\SSs \rm id}_i+U_iY^{\SSs \rm di}_i
\end{equation}
with correspondingly defined variables. As we sum up the $Y_i$ in \eqref{eqevent},
only $S_n=\sum Y_i$ will matter. As indicated in the outline, we split up the event $\{S_n>m\}$ by realizations of $K$, and
in the section $\{K=k\}$ we may suggestively write $S_n=S^{\SSs\rm id}_{n-k}+S^{\SSs \rm di}_{k}$, giving
$$
\{{\rm Med}_n\leq t\}={\Ts \dot\bigcup_{k=1}^m} \{S^{\SSs\rm id}_{n-k}+S^{\SSs\rm di}_{k}>m\}\cap\{K=k\}
$$
Splitting up again this event by the realizations of $S^{\SSs\rm di}_{k}$, we get
\begin{eqnarray}
\{{\rm Med}_n\leq t\}&=&{\Ts \dot\bigcup_{k=0}^m\dot\bigcup_{j=0}^k} \{S^{\SSs \rm id}_{n-k}>m-j\}\cap\{S^{\SSs\rm di}_{k}=j\}\cap\{K=k\}
\end{eqnarray}
Thus, for the moment, we may consider the situation that exactly $k$ observations, $0\leq k\leq m$, are contaminated, and
exactly $j=j_k(t)$ of the contaminated observations are larger than $t$ and denote that event with $D_{j,k,t}$.
As $\{X^{\SSs id}_{[m-j:n-k]}\leq t\}$ is independent from $D_{j,k,t}$, with $\bar F= 1-F$, the
the conditional density of ${\rm Med}_n$ knowing $D_{j,k,t}$ is
\begin{equation}
  g_{n,j,k}(t):=(n-k){{2m-k}\choose{m-j}} F(t)^{m-j}\bar F(t)^{m+j-k} f(t)\label{denseodd}
\end{equation}
Thus abbreviating again $j_k(t)$ by $j$, we get the following representation
\begin{equation}\label{nMSE}
  n\,{\rm MSE}({\rm Med}_n,G^{(n)})=n\, \sum_{k=0}^m \sum_{j=0}^k  \int t^2  g_{n,j,k}(t)\,dt\;P(S_k^{\SSs\rm di}=j)\,
  P(K=k)
\end{equation}

\subsection{Auxiliary results}
%
%
Before starting with the results we need some preparations
\subsubsection{Stirling approximations}\label{stirlsect}
We start with writing down some approximations for the factorials and the binomial coefficients derived from the Stirling formula
to be found e.g.\ in \citet[6.1.37]{Ab:St:84}:
\begin{eqnarray}
\!\!\!\!\!{{2n-k}\choose{n-j}}&=&(\Tfrac{2n-k}{\max(n-j,1)})^{n-j} (\Tfrac{2n-k}{n+j-k})^{n+j-k} \sqrt{\Tfrac{2n-k}{(n+j-k)(n-j)2\pi}} (1+\rho_{n,j,k}),\;\;\;
\mbox{for}\,-\Tfrac{1}{2}-\Tfrac{1}{48n}\leq\rho_{n,j,k}\leq \Tfrac{1}{12n},\\
\!\!\!\!\!&=& (\Tfrac{2n-k}{n-j})^{n-j} (\Tfrac{2n-k}{n+j-k})^{n+j-k} \sqrt{\Tfrac{2n-k}{(n+j-k)(n-j)2\pi}}(1-\Tfrac{1}{8n}+\Lo(\Tfrac{1}{n})),\;\;\;
\mbox{for } j,k=\LO(\sqrt{n\,}\,),
\label{chjkn}
\end{eqnarray}
The next lemma will be needed to settle case (III):
\begin{Lem}\label{binlem}
Let
\begin{equation}
\kappa:=k_1\log k_1+1-k_1
\end{equation}
Then
it holds that
\begin{equation}\label{alphahoeff}
\Pr({\rm Bin}(n,r/\sqrt{n\,}\,)>k_1 r\sqrt{n})\leq\exp\big(-\kappa\, r\sqrt{n}+\Lo(\sqrt{n}\,)\big)
\end{equation}
\end{Lem}
\begin{proof}{}
We first note that $\kappa>0$, as $\log(x)>0$ for $x>1$ and
$\kappa=\int_1^{k_1}\log(x)\, dx$.
By Hoeffding's inequality \citep[Thm.~1, inequality (2.1)]{Hoef:63}, we have for
$\xi_i$, $i=1,\ldots,n$ ${\rm i.i.d.}$ real--valued random variables, $|\xi_i|\leq M$,
$\mu=\Ew[\xi_1]$ and $0<\ve<1-\mu$
  \begin{eqnarray}
    P\big(\Tfrac{1}{n}\Tsum_i \xi_i -\mu \geq \ve \big) &\leq& \Big\{\left(\Tfrac{\mu}{\mu+\ve}\right)^{\mu+\ve}
    \left(\Tfrac{1-\mu}{1-\mu-\ve}\right)^{1-\mu-\ve}\Big\}^n\label{hoe3}
  \end{eqnarray}
Applying \eqref{hoe3} to the case of $n$ independent ${\rm Bin}(1,r/\sqrt{n\,}\,)$ variables,
we obtain for $B_n\sim{\rm Bin}(n,r/\sqrt{n\,}\,)$ and $0<\ve=(k_1-1)r/\sqrt{n}<1-r/\sqrt{n\,}\,$:
\begin{eqnarray*}
  \Pr(B_n>k_1r\sqrt{n}\,)&\leq& 
\exp\Big(-k_1r\sqrt{n}\,\log(k_1)+(n-k_1r\sqrt{n}\,)
 \big(\log(1-\frac{r}{\sqrt{n}\,})-\log(1-k_1\frac{r}{\sqrt{n}\,})\big)\Big)
\end{eqnarray*}
For $x \in(0,1)$,
$-\frac{x}{1-x} \leq  \log(1-x)\leq -x$.
Thus the difference of the logarithms is smaller than
$
(k_1r)/\big(\sqrt{n}\,(1-k_1r/\sqrt{n}\,)\big)-r/\sqrt{n}
$
and \vspace{1ex}\newline
\centerline{$\Pr(B_n>k_1r\sqrt{n}\,)\le
\exp\big(-\kappa\, r\sqrt{n}+\Lo(\sqrt{n}\,)\big)$}
\qed\end{proof}
%
\begin{Cor} \label{corewk}
Let $X\sim{\rm Bin}(n,r/\sqrt{n})$. Then for each $i\in\N_0$
\begin{equation}
\Ew[X^i \Jc_{\{X\geq k_1 r \sqrt{n}\}}]=\Lo(n^{-1})
\end{equation}
\end{Cor}
\begin{proof}{}
$\Ew[X^i \Jc_{\{X\geq k_1 r \sqrt{n}\}}]
\leq
n^i\Pr(X>k_1r\sqrt{n})\stackrel{\eqref{alphahoeff}}{\leq}{\rm const}\, n^i\exp(-\kappa\, r\sqrt{n})
$
\qed\end{proof}
\begin{Lem} \label{f0lem}
 We have that for $j,k=\LO(\sqrt{n})$
  \begin{equation}
|\Tfrac{m-j}{2m-k}-F(t)| \leq k_2 \, \sqrt{\frac{\log(n)}{n}}\,(1+\Lo(n^0))
\iff
|t|\leq \frac{k_2}{f_0}\,\sqrt{\frac{\log(n)}{n}}\,(1+\Lo(n^0))
  \end{equation}
\end{Lem}
\begin{proof}{} 
Using the fact that $j,k=\LO(\sqrt{n}\,)$, we note that
\begin{equation}
\Tfrac{m-j}{2m-k}=1/2+\Tfrac{k-2j}{4m}+\Tfrac{k(k-2j)}{8m^2}+\Lo(n^{-1})
\end{equation}
By \eqref{tayl1}, \eqref{medunb},
$F(t)=1/2+f_0t+\Lo(t)$; thus
$
|\Tfrac{m-j}{2m-k}-F(t)|=|\LO(\Tfrac{1}{\sqrt{n}}\,)-f_0t|
$.
\qed\end{proof}
\begin{Lem}\label{binmomlem}
  Let $X\sim {\rm Bin}(n,p)$. Then,
for $p=r/\sqrt{n}$,
  \begin{align}
    &\Ew[X]=rn^{1/2},\qquad\Ew[X^2]=r^2n+rn^{1/2}-r^2,\\
    &\Ew[X^3]=r^3n^{3/2}+3r^2n+(r-3r^3)n^{1/2}-3r^2+2r^3n^{-1/2},\\
    &\Ew[X^4]=r^4n^2+6r^3n^{3/2}+(7r^2-6r^4)n+(r-18r^3)n^{1/2}+
    +11r^4-7r^2+12r^3n^{-1/2}-6r^4n^{-1}
  \end{align}
\end{Lem}
\begin{proof}{}
Cf.\  the {\tt MAPLE}-procedure {\tt Binmoment} on the web-page.
\qed\end{proof}
Finally, we note  the following Lemma for ${\cal N}(0,1)$ variables
\begin{Lem} \label{lemnormlog}
  Let $X\sim {\cal N}(0,1)$. Then for $k\in\N$ and any $c>\sqrt{2}$,
  \begin{equation}
    \Ew [|X|^k\Jc_{\{|X|\geq c\sqrt{\log(n)}\}}] =\Lo(n^{-1})
  \end{equation}
\end{Lem}
\begin{proof}{} 
Let $\Phi(x):=\Pr(X\leq x)$, $\bar \Phi:=1-\Phi$, $\varphi(x)$ the density of $X$.
Then
$$
\Ew [X^k\Jc_{\{X\geq c\sqrt{\log(n)}\}}]=\left\{
\begin{array}{lcl}
P_k(x)\,\varphi(x)\Big|_{c\sqrt{\log(n)}}^{\infty}&\mbox{for}&\mbox{$k$ odd}\\
P_k(x)\,\varphi(x) +\prod_{i=1}^{k/2} (2i-1) \Phi(x)\Big|_{c\sqrt{\log(n)}}^{\infty} &&\mbox{$k$ even}\\
\end{array}
\right.
$$
%
%
for some polynomial $P_k$ of degree $k-1$. The assertion follows,
as $\varphi(c\sqrt{\log(n)})=\varphi(0)n^{-c^2/2}=\varphi(0)n^{-(1+\delta)}$ for some $\delta>0$,
and because for the $\Phi(x)$-term, $\bar\Phi(x)\leq \varphi(x)/x$ for $x>0$.
\qed\end{proof}
\subsection{Proof for odd sample size}
We recall the density $g_{n,j,k}$ from \eqref{denseodd}:
$$
  g_{n,j,k}(t):=(n-k){{2m-k}\choose{m-j}} F(t)^{m-j}\bar F(t)^{m+j-k} f(t)
$$
So the integrand of interest is $n\, t^2\,g_{n,j,k}(t)$. Applying the Stirling approximation \eqref{chjkn} to the constants, we get
\begin{equation}
{{2m-k}\choose{m-j}}=\Big(\frac{2m-k}{m-j}\Big)^{m-j}\Big(\frac{2m-k}{m-k+j}\Big)^{m-k+j}\gamma_{n,j,k}
\end{equation}
with
\begin{equation}\label{gammadef}
\gamma_{n,j,k}:= \sqrt{\Tfrac{2m-k}{(m+j-k)(m-j)2\pi}} \,(1+\rho_{m,j,k})
\end{equation}
for $\rho_{m,j,k}$ from \eqref{chjkn}.
As $F(t)^{m-j}\bar F(t)^{m+j-k}$ suggests an asymptotic decay, we will expand $g_{n,j,k}$ at the mode of $F(t)^{m-j}\bar F(t)^{m+j-k}$.
Differentiating, we easily get that
\begin{equation}\label{expungl}
F(t)^{m-j}\bar F(t)^{m+j-k}\leq \Big(\frac{m-j}{2m-k}\Big)^{m-j}\Big(\frac{m+j-k}{2m-k}\Big)^{m+j-k}
\end{equation}
with equality iff $t=x_{n,j,k}$ for
\begin{equation}
x_{n,j,k}:=F^{-1}(\Tfrac{m-j}{2m-k})
\end{equation}
%
%
%
Introducing
\begin{equation}
  \Delta F_{n,j,k}:=F(t)-\Tfrac{m-j}{2m-k}=F(t)-F(x_{n,j,k}),
\end{equation}
we see that
\begin{equation}
g_{n,j,k}(t)=(n-k)\gamma_{n,j,k} f(t) [1+\Tfrac{2m-k}{m-j\,}\Delta F_{n,j,k}]^{m-j}[1-\Tfrac{2m-k}{m+j-k}\,\Delta F_{n,j,k}]^{m+j-k}\label{convenii}
\end{equation}
\smallskip\\
{\bf Case (III)}: 
For $k\in [k_1 r\sqrt{n},\ve n]$, $0\leq j\leq k$, we partition the terms according to \eqref{convenii}
and see that on $|t|\leq n^2$, the integrand $n\,t^2\,g_{n,j,k}(t)$ multiplied by $n^{-4}$ is $\Lo(n^0)$ for each fixed $t$
and is dominated by $f(t)$ and hence by dominated convergence tends to $0$ as $n\to \infty$.
But Lemma~\ref{binlem} yields that $\Pr(K\geq k_1r\sqrt{n})$ decays exponentially in $n$, hence is even $\Lo(n^{-4})$, so  as noted, (III)
is indeed negligible asymptotically to order $\Lo(n^{-1})$.\smallskip\\
%
%
\noindent{\bf Case (II)}: 
Here $k\leq k_1 r\sqrt{n}$ and
$|t|>\Tfrac{k_2}{f_0} \sqrt{\log(n)/n}$, or equivalently by Lemma~\ref{f0lem}:
\begin{equation}\label{equivk2}
|\Delta F_{n,j,k}|>k_2 \sqrt{\log(n)/n}
\end{equation}
Now for $x>0$, $\log(1+x)\leq x$ and for $0<x<1$, $\log(1-x)\leq -x-x^2/2$.
Hence, we obtain
%
eventually in $n$
\begin{equation}
g_{n,j,k}(t)/f(t)\leq(n-k)\gamma_{n,j,k} \exp(-\Tfrac{(2m-k)^2}{2m} \Delta F_{n,j,k}^2)
\stackrel{\mbox{\tiny \eqref{chjkn}}\atop\mbox{\tiny \eqref{equivk2}}}{\leq}(n-k)
\sqrt{\Tfrac{1}{(m-k/2)}} (1+\Tfrac{1}{12m})
\exp[-\,k_2^2\log(m)
]\end{equation}
Plugging in that $m-k/2\geq m-k_1r\sqrt{m/2}$ in (II), we get
\begin{eqnarray}
g_{n,j,k}(t)/f(t)
&\leq &{\rm const}\,  m^{\frac{1}{2}-k_2^2}\big(1+\Lo(n^0)\big)=\Lo(n^{-2}) \label{condo}
\end{eqnarray}
where the last equality is a consequence of $k_2>\sqrt{5/2}$.
So negligibility 
follows 
by dominated convergence.
\smallskip\\
\noindent{\bf Case (IV)}: 
We only treat the case $t>n^2$; a corresponding relation holds for $t<-n^2$.
Under \eqref{contdef}, for $n$ large enough, we obtain bound
 $
g_{n,j,k} \leq n 2^n  f(t) \bar F(t)^{[(1-2\ve)n-1]/2}
 $.
Let $\eta=1/2-\ve$, $b=2/\delta$ and $\delta'\in (0,1)$.
By choosing $n$ large enough, we may achieve that $\bar F(n^2)=:\lambda_n<2^{-1/\eta}$ and $F(n^2)^{2b}>1-\delta'$.
So by \eqref{Jurec}, we get  eventually in $n$ and for some constant $c$ and any $\eta'>0$, and $g_\delta$ from the
proof of Remark~\ref{slowdecay}(a)
\begin{eqnarray*}
&&n\int_{n^2}^\infty t^2 g_{n,j,k}(t)\,dt \leq
\frac{n^2}{1-\delta'} 2^n  \int_{n^2}^\infty [t^\delta F(t)\bar F(t)]^b  [F(t)\bar F(t)]^b \,dt \lambda_n^{\eta n -1/2-2b} \leq \\
&\leq& \frac{\hat g_\delta^b I_b n^2}{(1-\delta')\lambda_n^{1/2+2b}} (2\lambda_n^\eta)^n =
c \exp(-|\log \lambda_n| n \big[\eta-\Tfrac{\log 2}{|\log \lambda_n|}-\Tfrac{1+4b}{2n}-\Tfrac{2\log n}{n \log \lambda_n}\big])\leq
\exp(-\eta' n)
\end{eqnarray*}
%
\smallskip\\
{\bf Case (I)}: 
Here we restrict ourselves to the case that
\begin{equation}\label{I-def}
k\leq k_1\sqrt{n}r,\qquad |\Tfrac{m-j}{2m-k}-F(t)|\leq k_2\sqrt{\log(n)/n}
\end{equation}
Doing so, we set
$u:=t-x_{n,j,k}$.
As on (I), $k=\LO(\sqrt{n})$ as well as $j$, we make
this magnitude explicit to {\tt MAPLE} in the function {\tt transf} by introducing the bounded variables
\begin{equation}\label{tildekl}
  \tilde k:=k/ \sqrt{m}\qquad\mbox{ and }\tilde \jmath :=(k/2-j)/ \sqrt{m}
\end{equation}
This gives the expansion in powers of $m^{-1/2}$
\begin{equation}
{(m-j)}/{(2m-k)}=1/2+
{\tilde \jmath}/{(4\sqrt{m})}+
{\tilde k \tilde \jmath}/{(8m)}+\Lo(n^{-1})
\end{equation}
Thus, to get an approximation to $x_{n,j,k}=F^{-1}(\Tfrac{m-j}{2m-k})$, we
expand this in a Taylor series in powers of $m^{-1/2}$ (compare our {\tt MAPLE}-procedure {\tt asquantile})
which gives 
\begin{equation}\label{xnk}
x_{n,j,k}=\Tfrac{\tilde \jmath}{2f_0\sqrt{m}}+\Tfrac{2f_0^2\tilde \jmath \tilde k-f_1 \tilde \jmath^2}{8f_0^3m}+
\Tfrac{6f_0^4\tilde \jmath \tilde k^2-6f_0^2f_1\tilde \jmath^2\tilde k-f_0f_2\tilde \jmath^3+3f_1^2\tilde \jmath^3}{48f_0^5m^{3/2}}+\Lo(n^{-3/2})
\end{equation}
Furthermore,
\begin{eqnarray*}
f(x_{n,j,k})
&=&f_0-%
\frac{f_1\tilde \jmath}/{(2f_0m^{1/2})}+
{\big(-f_1^2\tilde \jmath^2+2f_1f_0^2\tilde \jmath \tilde k+f_2f_0\tilde \jmath^2\big)}/{(8f_0^3m)}+\Lo(1/n)
\end{eqnarray*}
which implies that in (I), by \eqref{I-def}, $u$ lies in a shrinking compact, as
\begin{eqnarray*}
u&=&F^{-1}(F(t))-F^{-1}(F(x_{n,j,k}))=
f(x_{n,j,k})^{-1}(F(t)-\Tfrac{m-j}{2m-k})+\Lo(\sqrt{\log(n)/n})=\LO(\sqrt{\log(n)/n}).
\end{eqnarray*}
Setting $\Delta F_{n,j,k}:=F(t)-F(x_{n,j,k})$, and expanding this in a Taylor series around $0$, we get
\begin{eqnarray*}
\!  \Delta F_{n,j,k}&=&f_0u+f_1(u^2/2+u x_{n,j,k})+
  f_2(u^3/6+(u^2x_{n,j,k}+ux_{n,j,k}^2)/2)+
  \Lo(n^{-3/2})
\end{eqnarray*}
and
$$f(t)=f_0+f_1(u+x_{n,j,k})+f_2((u+x_{n,j,k}))^2/2+\Lo(n^{-1})$$
We turn to the constant factors now; up to now, the terms
arising by applications of the Stirling formulas of subsection~\ref{stirlsect} come with $k$--terms in the nominators.
As we want to integrate over $K$ later, however, it is preferable to move these terms into the denominators by Taylor approximations
---here performed by the functions {\tt asympt} and {\tt collect} in {\tt MAPLE} (compare our function {\tt asbinom}):%
\begin{eqnarray}
n(n-k)\sqrt{2\pi}\gamma_{n,j,k}
&=&2^{\frac{5}{2}} m^{3/2}[1-\Tfrac{\tilde k}{4m^{1/2}}+\Tfrac{16\tilde \jmath^2-\tilde k^2+28}{32m}]+\Lo(n^{\frac{1}{2}})
\label{mfac1}\\
\Tfrac{(2m-k)^2}{2(m-j)}+\Tfrac{(2m-k)^2}{2(m+j-k)}&=&4m
(1-\Tfrac{\tilde \jmath+\tilde k}{m^{1/2}}+\Tfrac{\tilde \jmath^2+\tilde \jmath \tilde k+\frac{\tilde k^2}{2}}{m})+\Lo(n)\\
\Tfrac{(2m-k)^3}{3(m-j)^2}-\Tfrac{(2m-k)^3}{3(m+j-k)^2}&=&
\Tfrac{16 \tilde k \sqrt{m}}{3}-(8\tilde k^2+16\tilde \jmath \tilde k)+\Tfrac{20 \tilde k^3+72\tilde \jmath \tilde k^2+96\tilde \jmath^2 \tilde k}{3m^{1/2}}
+\Lo(n^{-\frac{1}{2}})
\end{eqnarray}
Next we expand $[1+\Tfrac{2m-k}{m-j\,}\Delta F_{n,j,k}]^{m-j}[1-\Tfrac{2m-k}{m+j-k}\,\Delta F_{n,j,k}]^{m+j-k}$: We
plug in \eqref{xnk}, set
\begin{equation}
\sigma^2_n:=8mf_0^2,\qquad y:=u\sigma_n
\end{equation}
and apply the Taylor expansion $\exp(x)=1+x+x^2/2+\Lo(x^2)$. This gives
$$[1+\Tfrac{2m-k}{m-j\,}\Delta F_{n,j,k}]^{m-j}[1-\Tfrac{2m-k}{m+j-k}\,\Delta F_{n,j,k}]^{m+j-k}=\exp(-y^2/2)\, h(y,\tilde \jmath ,\tilde k,n)+\Lo(n^{-1})$$
with
\begin{eqnarray}
h(y,\tilde \jmath ,\tilde k,n)=1+\Big(\left({\Tfrac {\tilde k}{4}}-\,{\Tfrac {f_1\,\tilde \jmath\,{y}^{2}}{2{f_0}^{2}}}\right)\,y^2-\,
{\Tfrac {f_1}{8{f_0}^{2}}}\,\sqrt {2}y^3\Big)\,m^{-1/2}
+P(y,\tilde k,\tilde \jmath)m^{-1}
\end{eqnarray}
where $P$ is some polynomial depending on $f_0,f_1,f_2$ with ${\rm deg}(P)(y)=6$ the exact expression of which may
be drawn from the {\tt MAPLE}-script.
Accordingly, we define $\tilde x_{n,j,k}:=x_{n,j,k}\sigma_n$, and,
with $\varphi$ the density of ${\cal N}(0,1)$, use the abbreviations
\begin{equation}
\tilde \varphi(t) = \varphi\circ y \circ u(t),\qquad \tilde h(t,\tilde \jmath,\tilde k,n)=h (y \circ u(t),\tilde \jmath,\tilde k,n)
\end{equation}
We also introduce the integration domains
\begin{equation}
A_{n,j,k}=\Big\{t\in\R \,\Big|\,|\Tfrac{m-j}{2m-k}-F(t)|\leq k_2\sqrt{\Tfrac{\log(n)}{n}\,}\,\Big\},\qquad
\tilde A_{n,j,k}=\Big\{|t|\leq k_2\sqrt{\Tfrac{\log(n)}{n}}(1+\Lo(n^0))/f_0\Big\}
\end{equation}
Finally, applying \eqref{chjkn} and \eqref{mfac1}, we derive an integration constant $c_{n,j,k}$ from $\gamma_{n,j,k}$
from \eqref{gammadef}:
\begin{equation}
  c_{n,j,k}:=2^{-\frac{5}{2}}m^{-\frac{3}{2}}\gamma_{n,j,k}=
1-
{\tilde k}/{(4m^{1/2})}+
{(16\tilde \jmath^2-16\tilde \jmath \tilde k+3\tilde k^2+12)}/{(32m)}
\end{equation}
Plugging this all together, we obtain
\begin{equation}\nonumber
\int_{A_{n,j,k}} n\, t^2g_{n,j,k}(t)\,dt=(c_{n,j,k}+\Lo({\Tfrac{1}{n}}))\,\int_{\tilde A_{n,j,k}}
2^{\frac{5}{2}} m^{3/2} \; t^2f(t)\,\tilde \varphi(t)\,\tilde h(t,\tilde \jmath,\tilde k,n)\,dt
\end{equation}
%
Substituting $t(y)=\frac{y+\tilde x_{n,j,k}}{\sigma_n}$, we get
\begin{eqnarray*}
&&\int_{A_{n,j,k}} n\, t^2g_{n,j,k}(t)\,dt=
\int c_{n,j,k} \big(1+\Tfrac{f_1}{f_0}t(y)+\Tfrac{f_2}{f_0}
t(y)^2+\Lo(n^{-1})\big)\,
\varphi(y)\,h(y,\tilde \jmath,\tilde k,n)\Tfrac{(y+ \tilde x_{n,j,k})^2}{4f_0^2}\,
\Jc_{\tilde A_{n,j,k}}\big(t(y)\big) \,dy
\end{eqnarray*}
As $\tilde x_{n,j,k}=\LO(n^0)$,
\begin{equation}\nonumber
\Big\{|y+\tilde x_{n,j,k}|\leq 2k_2\sqrt{\log(n)}(1+\Lo(n^0))\Big\}=\Big\{|y|\leq 2k_2\sqrt{\log(n)}(1+\Lo(n^0))\Big\}=:A^0_{n}
\end{equation}
For the aggregation of the factors we use {\tt MAPLE}, giving
\begin{equation}\label{abschreq}
\int_{A_{n,j,k}} n\, t^2g_{n,j,k}(t)\,dt=
\int_{A^0_{n}} \Bigg[{\frac {(\,y+\sqrt{2}\,\tilde \jmath\,)^2}{4{{f_0}}^{2}}}+P_{1;n,\tilde \jmath, \tilde k}(y)m^{-1/2}+
P_{2;n,\tilde \jmath, \tilde k}(y)m^{-1}+\Lo(n^{-1})\Bigg] \varphi(y)\,dy
\end{equation}
for polynomials in $y$, $P_{1;n,\tilde \jmath, \tilde k}$ and $P_{2;n,\tilde \jmath, \tilde k}$
obtained by our {\tt MAPLE}-procedure {\tt getasintegrand}, where $P_{1;n,\tilde \jmath, \tilde k}$ is defined as
\begin{eqnarray*}
\Tfrac{y^2 \tilde k(y^2-1)}{16f_0^2}+\Tfrac{\sqrt{2}\,y^3f_1(2-y^2)}{32f_0^4}
+\left ( \Tfrac{\sqrt{2}y\,\tilde k(y^2+1)}{8f_0^2}
+\Tfrac{y^2f_1(3-2y^2)}{8f_0^4}\right )\tilde \jmath
+
\left (\Tfrac {(3+y^2)\tilde k}{8\,f_0^2}+
\Tfrac{(4-5\,{y}^{2})\sqrt {2}\,f_1\,y}{16f_0^4}\right )\,{\tilde \jmath}^{2}
-\Tfrac{ f_1 y^2}{4f_0^4}\,\tilde \jmath^3-\Tfrac{f_1 \sqrt{2}\,y} {4f_0^4}\,\tilde \jmath^4
\end{eqnarray*}
and $P_{2;n,\tilde \jmath, \tilde k}$ as
\begin{eqnarray*}
\lefteqn{P_{2;n,\tilde \jmath, \tilde k}(y)=
{\Tfrac {\tilde k^2(y^6-2y^4-y^2)+y^2(7-4y^4)}{128f_0^2}}+
{\Tfrac {{f_1}\tilde k\,\sqrt {2}{y}^{3}(-2+5y^2-y^4)}{128f_0^4}}+{\Tfrac {f_2y^4(3-y^2)}{192f_0^5}}+\Tfrac {f_1^2y^6(y^2-5)}{256f_0^6}+}\\
&&\qquad+\hphantom{\big[}{
\Big({\Tfrac {\sqrt{2}\big(\tilde k^2y(9+6y^2+3y^4)+28y(3-y^4)\big)}{192\,f_0^2}}+
{\Tfrac {\tilde k\, f_1\,y^2\,(-2y^4+5y^2+3)}{32\,f_0^4}}+
\Tfrac {\sqrt{2}\,y^3f_2(12-5y^2)}{192\,f_0^5}}+
{\Tfrac {\sqrt{2}\,y^5f_1^2(3y^2-13)}{128\,f_0^6}}
\Big )\,\tilde \jmath
+\\
&&\qquad+\hphantom{\big[}
\Big(
\,{\Tfrac {\tilde k^2(45+18y^2+3y^4)-100y^4-24y^2+84}{192f_0^2}}+
{\Tfrac {{f_1}\,\tilde k\,\sqrt {2}y(12-y^2-5y^4)}{64f_0^4}}+
{\Tfrac {y^2f_2(9-5y^2)}{48f_0^5}}
+{\Tfrac {{{f_1}}^{2}{y}^{2}(13y^4-41y^2-12)}{128f_0^6}}\Big )
\,{\tilde \jmath}^{2}
+\\
&&\qquad+\hphantom{\big[}\Big ({\Tfrac {\sqrt{2}y(3-5y^2)}{12f_0^2}}-{\Tfrac {\tilde k{f_1}y^2\,({y}^{2}+3)}{16f_0^4}}+
{\Tfrac {f_2\sqrt{2}y(10-9y^2)}{96f_0^5}}+{\Tfrac {f_1^2\sqrt{2}y(3y^4-5y^2-4)}{32f_0^6}}
\Big )\,{\tilde \jmath}^{3}+
\Big (\Tfrac{1-y^2}{4{f_0}^2}+{\Tfrac {{f_2}(1-3y^2)}{48\,{{f_0}}^{5}}}
+{\Tfrac {{f_1}^2\,(2y^4-1)}{32\,{{f_0}}^{6}}}\Big )\,{\tilde \jmath}^{4}
\end{eqnarray*}
By the restriction in $A_n^0$, we obtain that
$|y|=\LO(\sqrt{\log(n)})$, while ${\rm deg}(P_{1;\cdot};y)=5$ and ${\rm deg}(P_{2;\cdot};y)=8$.
Hence, the integrand is apparently of form
$
{(y+\sqrt{2}\,\tilde \jmath)^2}/{(4{{f_0}}^{2})}+\LO(\sqrt{\log(n)^5/n}\,)$,
and thus, eventually in $n$,  is maximized---up to $\LO(\sqrt{\log(n)^5/n}\,)$---for $|\tilde \jmath|$
maximal, i.e.\ $|\tilde \jmath|=\tilde k/2$.
Even more so, if $f_1=0$, the $m^{-1/2}$--term, too, is maximized for $|\tilde \jmath|=\tilde k/2$. As the highest
power in $P_{2;\cdot}$ occurring to $y$ in a $\tilde \jmath$--term without $f_1$ is $4$,
the integrand is maximized up to $\LO((\log(n)^4/n))$ for $|\tilde \jmath|=\tilde k/2$.\\
Condition $|\tilde \jmath|=\tilde k/2$ is equivalent to $j_k(t)\equiv k$ or $j_k(t)\equiv0$.
But this is the case---up to $\Lo(n^{-1})$---if condition \eqref{contbed1} or \eqref{contbed2} is in force, as then
up to mass of order $\Lo(n^{-1})$ the contamination is either concentrated left of $-\Tfrac{k_2}{f_0}\sqrt{\log(n)/n}$ or
right of $\Tfrac{k_2}{f_0}\sqrt{\log(n)/n}$ for any sample with no more than $k_1r\sqrt{n}$ contaminations.
With respect to \eqref{condo}, this suffices to obtain that (II) is $\Lo(n^{-1})$.\\
Later, after having integrated out $y$, we will see that if $f_1=0$, the approximation up to order $n^{-1}$ is identical
for $j_k(t)\equiv k$ and $j_k(t)\equiv 0$, whereas if $f_1>0$ it pays off for nature to contaminate by positive values and,
correspondingly, by negative values if $f_1<0$. We consider $j_k(t)\equiv k$ here.\\ 
Up to $\Lo(n^{-1})$,  $\int_{A_{n,k,k}} n\, t^2g_{n,k,k}(t)\,dt$  is
\begin{eqnarray}
\int_{A_n^0} \Big[\;\Tfrac{1}{4f_0^2}(y^2+\tilde k^2/2)+
Q_{1;n,\tilde k}(y) m^{-1/2} + Q_{2;n,\tilde k}(y)m^{-1}+
\tilde g(n,k,y)\Big] \varphi(y)\,dy \label{yeq}
\end{eqnarray}
with some skew-symmetric polynomial $\tilde g$  in $y$ of degree $5$ that 
is uniformly bounded in $n$ on $A_n^0$,
and for some even-symmetric polynomials $Q_{1;n,k}(y)$ and $Q_{2;n,k}(y)$; we only present the  definition of $Q_{1;n,k}(y)$ below;
for $Q_{2;n,k}(y)$, we refer the reader to the corresponding {\tt MAPLE}-procedure {\tt getasrisk}.
\begin{eqnarray*}
Q_{1;n,\tilde k}(y)&=&\Tfrac{3\tilde k^3}{32f_0^2}+\left(\Tfrac{f_1( \tilde k^3-3\tilde k )}{32f_0^4}+
\Tfrac{\tilde k^3-2\tilde k}{32f_0^2}\right)y^2+
\left(\Tfrac{\tilde k}{16f_0^2}+\Tfrac{f_1\tilde k}{8f_0^4}\right)y^4
\end{eqnarray*}
Using Lemma~\ref{lemnormlog} we see that we may drop the restriction
$|y|\leq 2k_2\sqrt{\log(n)}$ and integrating $y$ out, up to $\Lo(n^{-1})$,  we get that $\int_{A_{n,k,k}} n\, t^2g_{n,k,k}(t)\,dt$ is
\begin{eqnarray*}
&&\Tfrac{1+\tilde k^2/2}{4f_0^2}+\Big[(\Tfrac{1}{8f_0^2}-\Tfrac{3f_1}{16f_0^4})\,\tilde k+(\Tfrac{1}{8f_0^2}+\Tfrac{f_1}{32 f_0^4})\,\tilde k^3
 \Big]m^{-1/2}+\Big[(\Tfrac{-1}{4f_0^2}-\Tfrac{f_2}{32f_0^5}+\Tfrac{15f_1^2}{128f_0^6})+\\
&&\quad+(\Tfrac{-3}{16f_0^2}+\Tfrac{3f_1}{16f_0^4}-\Tfrac{f_2}{32f_0^5}+\Tfrac{15f_1^2}{128f_0^6})\,\tilde k^2+
(\Tfrac{3}{32f_0^2}-\Tfrac{f_2}{384f_0^5}+\Tfrac{3f_1}{64f_0^4}-\Tfrac{5f_1^2}{512f_0^6})\,\tilde k^4\Big]m^{-1}
\end{eqnarray*}
Corollary~\ref{corewk} gives that we may ignore the fact that $k$ is restricted to $k\leq k_1r\sqrt{n}$ and so with Lemma~\ref{binmomlem},
we may simply integrate out $k$. After substituting $n=2m+1$ we thus indeed get
\begin{eqnarray}
&&  \sup_{G^{(n)}}\,n\,[{\rm MSE}({\rm Med}_n,G^{(n)})]=\frac{1}{4 f_0^2}\Bigg\{
(1+r^2)+\Tfrac{r}{\sqrt{n}}\Big({\Ss 2(1+r^2)}+\Tfrac{f_1(r^2+3)}{2f_0^2}\Big)+\nonumber\\
&&\qquad +\Tfrac{1}{n}\Big({\Ss \big(3r^4+3r^2-2\big)}+\Tfrac{3r^2f_1(3+r^2)}{2f_0^2}-\Tfrac{f_2(r^4+6r^2+3)}{12f_0^3}+\Tfrac{5 f_1^2(r^4+6r^2+3)}{16f_0^4}\Big)\Bigg\}
+\Lo(n^{-1})\nonumber
\end{eqnarray}
Considering both cases $j_k(t)\equiv k$ and $j_k(t)\equiv 0$ simultaneously, we get \eqref{Mall} with \eqref{Malla1} and \eqref{Malla2}.\qed
%
\subsection{Proof of Proposition~\ref{medeventhm}---pure quantiles and randomization}
The proof for the pure quantiles is just as in  the odd case and thus skipped.
%
We only draw the attention to the different behaviour of the $1/\sqrt{n}$-correction
term for positive and negative contamination
which explains \eqref{s-def} in this case. 
For the bias corrected version $M''_n$, 
with the same techniques as in the proof of Theorem~\ref{medthm}, we calculate the bias
of $\sqrt{n}\,X_{[(m+1):n]}$ under $F$.
This gives $
 \sqrt{n}\,\Big|{\rm Bias}(X_{[m+1:n]},F^n)\Big|=B_{n,1}+B_{n,2}
$
for
\begin{equation}
  B_{n}=B_{n,1}+B_{n,2},\qquad B_{n,1}=\frac{1}{2f_0\sqrt{n}}, \qquad |B_{n,2}|=\frac{|f_1|}{8f_0^3\sqrt{n}}
\end{equation}
The same terms but with different signs are obtained for $\sqrt{n}\,\big|{\rm Bias}(X_{[m:n]},F^{n})\big|$.
We only consider $B_{n,1}$ here, which arises no matter if we have symmetry or not and gives the bias corrected version $M''_n$
with the $a_{i,j}$ terms as in Proposition~\ref{aijterms}.
\begin{Rem}\rm\small \small
  We note that in all variants of the sample median up to now a minor deterministic improvement is possible if $f_1\not=0$,
when we consider the bias-corrected estimators
\begin{equation}
M_n^\flat:=M_n-\frac{1}{\sqrt{n}}B_{n,2}=M_n+\frac{f_1}{8f_0^3n}
\end{equation}
Except for the pure quantiles for even $n$, this renders all variants bias--free up to $\Lo(n^{-1})$ in the ideal model.
\end{Rem}

\subsection{Proof of Proposition~\ref{medeventhm}---the midpoint-estimator}
For the midpoint--estimator $\bar M_n$, we need the common law of the pure quantile estimators $X_{[m:n]}$
and $X_{[(m+1):n]}$. So more generally, we start with the common law of
$(Y,Z):=(X_{[\nu_1:n]},X_{[\nu_2:n]})$ for $1\leq \nu_1<\nu_2\leq n$, $X_i\iid F$, $i=1,\ldots,n$ and $F(dx)=f(x)\,dx$,
see \citet[pp.~9--10]{Dav:70},
and in our case ($n\hat=2m$, $\nu_1\hat=m$, $\nu_2\hat=m+1$) 
leads us to  
the density of the midpoint estimator $\bar M_{2m}=(Y+Z)/2$
\begin{equation}
g_{n}(t)=(2m)^2\, {2m-1 \choose m} \int_t^{\infty}
\,\big[F(2t-u)\,\big(1-F(u)\big)\big]^{(m-1)}\,f(u)\,f(2t-u)\,du\label{denseevenid}
\end{equation}
This gives for $(2m)\,{\rm MSE}(\bar M_{2m},F)$, after substituting $s=2t-u$, and using Fubini
\begin{eqnarray}
\!&&(2m)\,{\rm MSE}(\bar M_{2m},F)=\nonumber\\
\!&&2m^2\, {2m-1 \choose m}  \int\int_{-\infty}^{u}
\! \frac{(u+s)^2}{4} \,\big[F(s)\,\big(1-F(u)\big)\big]^{(m-1)}\,f(u)\,f(s)\,ds\,du
\end{eqnarray}
We skip the argument showing how to choose a risk maximizing contamination.  
In the {\tt MAPLE} script, however, we have detailed out a corresponding argument for
$j(t)$ the number of contaminated observations larger than $t$. 
Without loss of generality, we work with the case of contamination to the right.
Analogue arguments as in the preceding cases show that given we have $k$ observations
contaminated to $\infty$, we get as  expression for the (conditional) ${\rm MSE}_{|K=k}$:
\begin{eqnarray}
(2m)\,{\rm MSE}_{|K=k}&=&{(2m)(2m-k)}\,
{\Ss 2m-1-k \choose \Ss m-k} \int F(u)^{(m-k)}\big(1-F(u)\big)^{(m-1)}\,f(u) \times \nonumber\\
&&\;\times \int_{-\infty}^{u} \!\Tfrac{(m-k)(u+s)^2}{4F(u)} \,\big(1-\frac{F(u)-F(s)}{F(u)}\big)^{(m-1-k)}\,f(s)\,ds\,du \label{denseeven}
\end{eqnarray}
which we have written in a way to be able parallel the preceding subsections.
Denote the value of the inner integral by $H_k(u)$ and
\begin{equation}
  \Delta(s,u):=(F(u)-F(s))/F(u)
\end{equation}
In the inner integral, $0\leq \Delta(s,u)\leq 1$, and  for $\Delta(s,u)>\alpha>0$,  $H_k(u)$
 will decay exponentially while being dominated, so if we introduce
\begin{equation}
  \delta(u):=\sup\big\{ s <u\, \big|\, F(s)\leq (1-\alpha)\,F(u)\big\}
\end{equation}
in fact we may restrict the inner integral to
\begin{equation}
H_k(u)=\Lo(m^{-1})+ \frac{m-k}{4F(U)} \int_{\delta(u)}^{u} \!(u+s)^2 \,\big(1-\Delta(s,u)\big)^{(m-1-k)}\,f(s)\,ds \label{Hueq}
\end{equation}
But then expanding $\log(1-\Delta(s,u))$, and in order to get the right order for the expansion substituting
$u=\tilde u/\sqrt{m}$, $s=\tilde s/\sqrt{m}$---according to case (I), i.e.; $|u|\leq {\rm const}\,\sqrt{\log(m)/m}$.
Thus, for polynomials $\bar Q_i$ in $\tilde s,\tilde u$ defined in analogy to the $Q_i$ in the to odd-sample case and with may
be looked up in the {\tt MAPLE} script,
$$\Delta(s,u)=
2\,{\Tfrac {f_0\,\left (\tilde s- \tilde u\right )}{\sqrt {m}}}+ \Tfrac{\bar Q_{0}(\tilde s,\tilde u)}{m\,}+
\Tfrac{\bar Q_{1}(\tilde s,\tilde u)}{m^{3/2}\,}+\Tfrac{\bar Q_2(\tilde s,\tilde u)}{m^2}+\LO(\big(\Tfrac{\log(n)}{n}\big)^{5/2})
$$
Hence  we get
\begin{equation}
{(m-k-1)}\log(1-\Delta(s,u)) -{\sqrt{m}}({\Ss 2\,f_0\,\left (\tilde s- \tilde u\right )})=
{\tt logH21}(s,u)+\Lo(\sqrt{\log(n)/n}\,) \label{logH21}
\end{equation}
for some function ${\tt logH21}$, the exact expression of which may be produced in the corresponding {\tt MAPLE} script.
Thus, denoting the term $\exp({2\sqrt{m}\,f_0\,(\tilde s- \tilde u)})$ by $H_{\,;1}(s,u)$, we get
$$
(1-\Delta(s,u))^{(m-k-1)}=H_{\,;1}(s,u) \exp({\tt logH21})
\times(1+\Lo(\sqrt{\log(n)/n})),
$$
Now, if we write $H_{\,;2,2}(s,u)$ for $(s+u)^2\,f(s)$, and
$H_{k;2,2}(s,u)$ for $\exp({\tt logH21})$, and if we introduce
$H_{k;2}(s,u):=H_{\,;2,1}(s,u)H_{k;2,2}(s,u)$, we get
$$
4 F(u)\,H_k(u)=\Lo(n^{-2})+\int_{\delta(u)}^u H_{\,;1}(s,u)\,H_{k;2}(s,u)\,ds
$$
The next step is to integrate out $s$ where we may drop the lower restriction again due to the exponential decay
far out for large values of $s$. After three times of integration by parts we come up with
\begin{equation}
4 F(u)\,H_k(u)=\Lo(n^{-2})+
\sum_{i=0}^2 \frac{(-1)^i}{(2\sqrt{m}\,f_0)^{(i+1)}}H_{\,;1}(s,u) \frac{\partial^i}{\partial s^i}H_{k;2}(s,u)\,\Big|_{-\infty}^u \label{Husum}
\end{equation}
that is we may restrict ourselves to these terms for our purposes. These differentiations can be done by the {\tt MAPLE}
command {\tt diff}.
%
Noting that essentially $t=\LO(\sqrt{\log(n)/n})$,
we hence get for the inner integral $H$
$$
H_k(t)=t^2+\Tfrac{1}{n}[(\Tfrac{f_1}{2f_0}-1)t^2-\Tfrac{1}{2f_0}t ]-\Tfrac{1}{\sqrt{n^3}}\Tfrac{k}{2f_0} t +
\Tfrac{1}{n^2}\Tfrac{1}{8f_0^2}+\Lo(n^{-2})
$$
So in formula~\eqref{nMSE} (with $j\equiv k$) we replace $t^2$ by $H_k(t)$ and arrive at
\begin{equation}\label{n2MSE}
  \sup_{G^{(n)}} n\,{\rm MSE}(\bar M_n,G^{(n)})=n\, \sum_{k=0}^m   \int H_k(t)\,  g_{n,k,k}(t)\,dt\;\,
  P(K=k) +\Lo(n^{-1})
\end{equation}
Proceeding now just as in the preceding subsections, we obtain the assertion. 
\subsection{Proof of Proposition~\ref{neccond}}\label{neccondp}
For $t>\sqrt{\log(n)/n}/(2f_0)$, let
\begin{equation}
A_{k,t}:=\Big\{\sum_{i}U_i\big(2\Jc(X_i\leq t\,)-1) \leq k-1\Big\}
\end{equation}
Hence if $t>\sqrt{\log(n)/n}/(2f_0)$, by \eqref{condnec},
for all $k>(1-\delta)r\sqrt{n}$,
\begin{equation}\label{restri}
\Pr(A_{k,t}\,\Big|\,K=k)\ge p_0
\end{equation}
Now we proceed as in the proof to Theorem~\ref{medthm}.
But $t>\sqrt{\log(n)/n}/(2f_0) \iff y> \sqrt{\log n }$ in \eqref{yeq}. Hence on the event $A_{k,t}$
for $y\in [ \sqrt{\log n };k_2\sqrt{\log n\,}\,)$, we get the bound  $\tilde \jmath(t)\leq (k-1)/\sqrt{n}$,
while for $y\in (-k_2\sqrt{\log(n)};\sqrt{\log n\,}\,)$ respectively on ${}^cA_{k,t}$, we  bound $\tilde \jmath(t)$ by $k/\sqrt{n}$.
Integrating out these two $y$-domains separately, we obtain
\begin{eqnarray*}
&&  n\,\Big({\rm MSE}[{\rm Med}_n,G_0^{(n)}\,\big|K=k\,]-{\rm MSE}[{\rm Med}_n,G_\flat^{(n)}\,\big|K=k\,]\Big)\geq \\
&\geq&\frac{p_0}{2f_0} \int_{\sqrt{\log n\, }}^{k_2 \sqrt{\log n\, }}\Big(s /\sqrt{n}+\tilde k/\sqrt{2n}-1/(2f_0n)\Big)\,\varphi(s)\,ds+\Lo(n^{-1})
\end{eqnarray*}
But for $0<a_1<a_2<\infty$,
$
\varphi(a_1)/a_2-\varphi(a_2)/a_2\le \int_{a_1}^{a_2} \,\varphi(s)\,ds
$,
so that with $a_1=2\sqrt{\log n }$, $a_2=k_2 \sqrt{\log n\, }$, and as $\varphi(a_2)=\Lo(n^{-1})$,
\begin{eqnarray*}
  n\,\Big({\rm MSE}[{\rm Med}_n,G_0^{(n)}\,\big|K=k\,]-{\rm MSE}[{\rm Med}_n,G_\flat^{(n)}\,\big|K=k\,]\Big)
\geq\frac{p_0}{2\sqrt{2\pi}\, n f_0}+\Lo(n^{-1})
\end{eqnarray*}
By Lemma~\ref{binlem}, the restriction to $(1-\delta)r\sqrt{n}<K<k_1r\sqrt{n}$ may be dropped, and we obtain the assertion.
The case of an even sample size is proved similarly.
\hfill\qed
\section*{Extra Material}
On [site of the journal] we have additional supplementary material for this article:
This comprises extended tables, details to points alluded to in remarks, but in particular
a {\tt MAPLE} script, referred to in the proofs. 
\begin{figure}[p]
\centerline{\includegraphics[width=13cm]{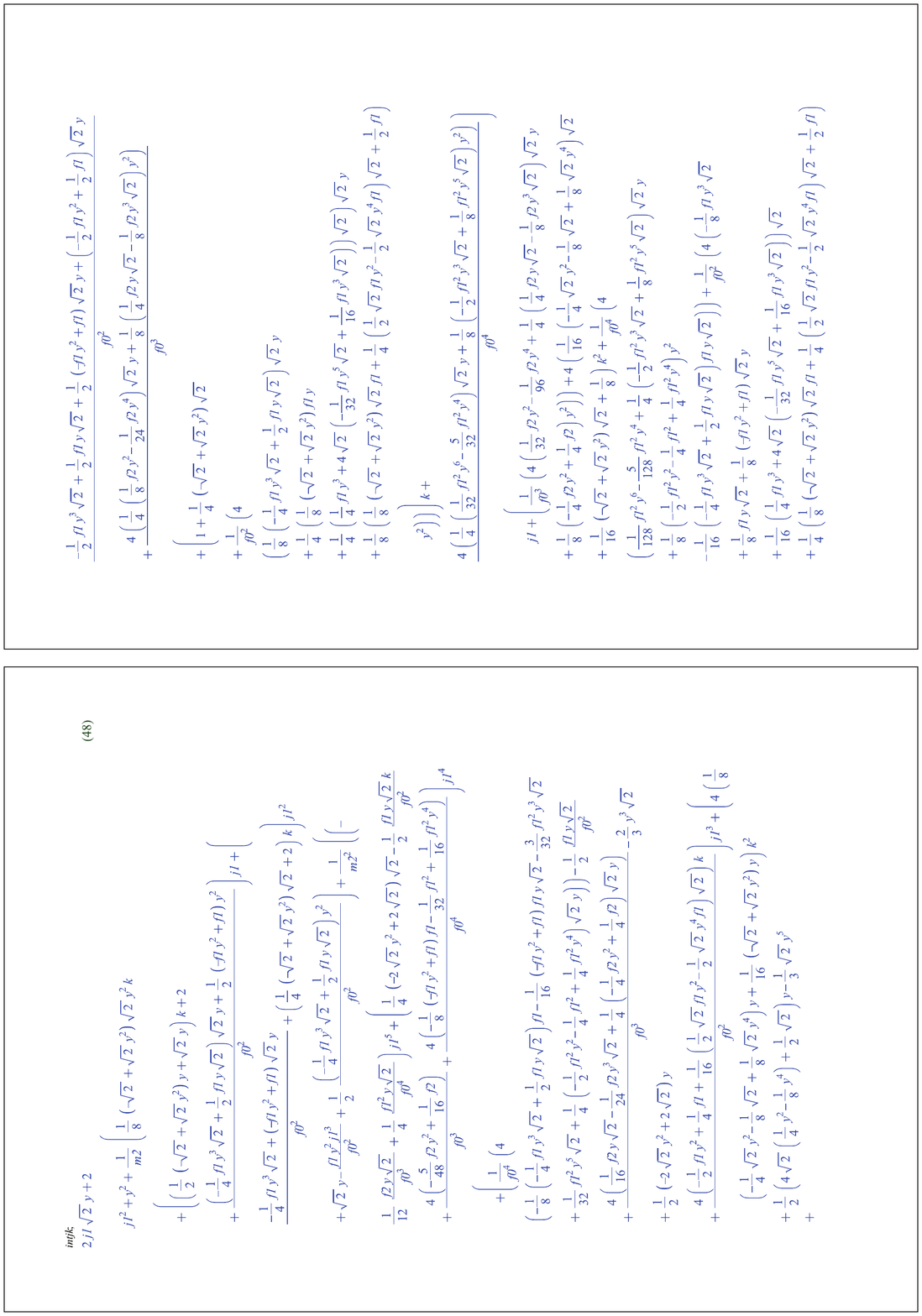}}
\caption{\label{abschrFig}A horrifying example: (The first two pages of) the expression for \eqref{abschreq} got from {\tt MAPLE};
of course, after integration terms get much more treatable, as visible  in Theorem~\ref{medthm}.
}
\end{figure}

\section*{Acknowledgement}
Thanks are due to J. Picek for drawing the author's attention to \citet{Jur:82}.

\medskip
\hrulefill\hspace*{6cm}\\[2ex]
Web-page to this article:\\
\href{http://www.mathematik.uni-kl.de/~ruckdesc/}%
{{\footnotesize \url{http://www.mathematik.uni-kl.de/~ruckdesc/}}}

\end{document}